\documentclass[12pt,a4paper]{amsart}
\usepackage{amsmath}
\usepackage{amsthm}
\usepackage{enumerate}
\usepackage{amssymb}
\usepackage[utf8]{inputenc}
\usepackage{color}
\usepackage{tikz}
\usepackage{verbatim}
\usepackage{hyperref,aliascnt}
\usepackage[all,cmtip,arrow,matrix,curve]{xy}

\theoremstyle{plain}

\newtheorem{lemma}{Lemma}[section]

\newaliascnt{thmCt}{lemma}
\newtheorem{theorem}[thmCt]{Theorem}
\aliascntresetthe{thmCt}

\newaliascnt{corCt}{lemma}
\newtheorem{corollary}[corCt]{Corollary}
\aliascntresetthe{corCt}

\newaliascnt{prpCt}{lemma}
\newtheorem{proposition}[prpCt]{Proposition}
\aliascntresetthe{prpCt}

\newtheorem*{proposition*}{Proposition}
\newtheorem*{definition*}{Definition}

\newcounter{theoremintro}
\newtheorem{thmintro}[theoremintro]{Theorem}

\theoremstyle{definition}

\newaliascnt{dfnCt}{lemma}

\aliascntresetthe{dfnCt}

\newaliascnt{rmkCt}{lemma}
\newtheorem{remark}[rmkCt]{Remark}
\aliascntresetthe{rmkCt}

\newaliascnt{rmksCt}{lemma}

\aliascntresetthe{rmksCt}

\newaliascnt{qstCt}{lemma}

\aliascntresetthe{qstCt}

\newaliascnt{exaCt}{lemma}
\newtheorem{example}[exaCt]{Example}
\aliascntresetthe{exaCt}

\newaliascnt{exasCt}{lemma}
\newtheorem{examples}[exasCt]{Examples}
\aliascntresetthe{exasCt}

\newaliascnt{paraCt}{lemma}
\newtheorem{parag}[paraCt]{}
\aliascntresetthe{paraCt}

\newcommand{\Cs}{{$\mathrm{C}^*$-al\-ge\-bra}}

\newcommand{\id}{\mathrm{id}}

\newcommand{\WLambda}{\Lambda_{\mathrm W}}
\newcommand{\RLambda}{\Lambda}

\newcommand{\preS}{{\mathcal S}}
\newcommand{\SR}{{\mathrm S}}

\newcommand{\preCP}{\mathcal{CP}}
\newcommand{\CP}{\mathrm{CP}}
\newcommand{\CH}{\mathrm{CH}}
\newcommand{\preCH}{\mathcal{CH}}

\newcommand{\W}{{\mathrm W}}
\newcommand{\V}{{\mathrm V}}

\newcommand{\Cu}{\mathrm{Cu}}
\newcommand{\SCu}{\mathrm{SCu}}
\newcommand{\mvns}{\precsim_{\textrm{MvN}}}

\newcommand{\Z}{{\mathbb Z}}
\newcommand{\N}{{\mathbb N}}

\newcommand{\ol}{\overline}

\title{The Cuntz semigroup of a ring}
\date{\today}

\author[R.~Antoine, P.~Ara, J.~Bosa, F.~Perera, and E.~Vilalta]{Ramon Antoine\and 
	Pere Ara\and
	Joan Bosa\and
	Francesc Perera\and
	Eduard Vilalta}

\address{Ramon Antoine,
	Departament de Matem\`{a}tiques,
	Universitat Aut\`{o}noma de Bar\-ce\-lo\-na,
	08193 Bellaterra, Barcelona, Spain, and
	Centre de Recerca Matem\`atica, Edifici C, Campus de Bellaterra, 08193 Bellaterra, Barcelona, Spain}
\email[]{ramon.antoine@uab.cat}

\address{Pere Ara, 
	Departament de Matem\`{a}tiques,
	Universitat Aut\`{o}noma de Bar\-ce\-lo\-na,
	08193 Bellaterra, Barcelona, Spain, and
	Centre de Recerca Matem\`atica, Edifici C, Campus de Bellaterra, 08193 Bellaterra, Barcelona, Spain}
\email[]{pere.ara@uab.cat}

\address{Joan Bosa,
Departamento de Matem\'{a}ticas,
Universidad de Zaragoza,
50009 Zaragoza, Zaragoza, Spain.}
\email{jbosa@unizar.es}
\urladdr{personal.unizar.es/jbosa/}

\address{Francesc Perera, 
	Departament de Matem\`{a}tiques,
	Universitat Aut\`{o}noma de Bar\-ce\-lo\-na,
	 08193 Bellaterra, Barcelona, Spain, and
	Centre de Recerca Mate\-m\`a\-tica, Edifici C, Campus de Bellaterra,  08193 Bellaterra, Barcelona, Spain}
\email[]{francesc.perera@uab.cat}
\urladdr{https://mat.uab.cat/web/perera}

\address{Eduard Vilalta, 
The Fields Institute for Research in Mathematical Sciences, M5T 3J1 Toronto, Canada}
\email[]{evilalta@fields.utoronto.ca}
\urladdr{www.eduardvilalta.com}

\thanks{All authors were partially supported by MINECO (grant No.\ PID2020-113047GB-I00), and by the Comissionat per Universitats i Recerca de la Generalitat de Catalunya (grants No.\ 2017SGR01725 and 2021SGR01015). The last named author was also supported by MINECO grant No.\ PRE2018-083419 and by the Fields Institute for Research in Mathematical Sciences.}

\subjclass[2020]%
{Primary
16D10, 
16B99, 
06F05; 
Secondary
46L05. 
}

\begin{document}

\begin{abstract}
For any ring $R$, we introduce an invariant in the form of a partially ordered abelian semigroup $\SR(R)$ built from an equivalence relation on the class of countably generated projective modules. We call $\SR(R)$ the Cuntz semigroup of the ring $R$. This construction is akin to the manufacture of the Cuntz semigroup of a C*-algebra using countably generated Hilbert modules. To circumvent the lack of a topology in a general ring $R$, we  deepen our understanding of countably projective modules over $R$, thus uncovering new features in their direct limit decompositions, which in turn yields two equivalent descriptions of $\SR(R)$. The Cuntz semigroup of $R$ is part of a new invariant $\SCu(R)$ which includes an ambient semigroup in the category of abstract Cuntz semigroups that provides additional information. We provide computations for both $\SR(R)$ and $\SCu(R)$ in a number of interesting situations, such as unit-regular rings, semilocal rings, and in the context of nearly simple domains. We also relate our construcion to the Cuntz semigroup of a C*-algebra.
\end{abstract}

\maketitle
\date{\today}

\section{Introduction}

The study of a ring using the collection of its projective (right) modules is central in modern algebra. Much attention has been directed to finitely generated projective modules, mostly with K-Theory in mind, since for a unital ring $R$ the Grothendieck group $\mathrm{K}_0(R)$ is constructed out of the monoid $\V(R)$ of isomorphism classes of such modules. There has also been an intensive use of countably generated projective modules. This may be justified keeping in mind that, by a well known theorem of Kaplansky, any projective module is a direct sum of countably generated projective ones. In this case, one might use the monoid $\V^*(R)$ of isomorphism classes of countably generated projective modules to analize the ring $R$. It is worth noticing that both monoids $\V(R)$ and $\V^*(R)$ are naturally equipped with the so-called algebraic order, given by complements.

The structure of $\V^*(R)$ has attracted considerable attention in the last years; see, for instance, \cite{HPCrelle} and \cite{HPTrans}. For a semilocal ring, and following a result obtained in \cite{Prihoda2007}, one has that $\V^*(R)$ can be viewed as a submonoid of $\V^*(R/J(R))$, which in turn is isomorphic to $\overline{\N}^r$ for a suitable $r$, where $\overline{\N}=\N\cup\{\infty\}$ with the obvious operations. A relevant problem is then to determine which submonoids of $\overline{\N}^r$ are realized by semilocal rings. A full characterization of such submonoids in the noetherian case is obtained in \cite[Theorem 2.6]{HPCrelle}, as those submonoids defined by a system of equations in the sense of \cite[Definition 2.5]{HPCrelle}. (We note that previous results for semilocal rings, but for the monoid $\V(R)$ and full affine submonoids of $\N^r$, were already obtained in \cite{FH2000}.) Further progress was carried out in \cite[Theorem 1.6]{HPTrans} for not necessarily noetherian semilocal rings. There, the authors studied countably generated projective modules that are finitely generated modulo the Jacobson radical and showed that they appear in a wide variety of situations.

Our aim here is to introduce an object $\SR(R)$ that can also be built out of countably generated projective modules, albeit using a  relation weaker than isomorphism. Its construction is inspired by that of the Cuntz semigroup of a C*-algebra, as the latter possesses a rich ordered structure that extends the algebraic order, and has played an important role in the theory of C*-algebras in recent years. Let us review this construction in relation to the main theme of this paper.

For the class of C*-algebras, that is, self-adjoint, norm-closed subalgebras of the algebra of bounded operators on a Hilbert space, we encounter in countably gene\-ra\-ted Hilbert modules the  analytic siblings  of countably gene\-ra\-ted projective modules. Roughly speaking, a Hilbert module over a C*-algebra $A$ is an $A$-module, together with an $A$-valued inner product, which is complete with respect to a suitable norm. The $A$-module $A^{(\N)}$, with a natural inner product, gives rise to the standard Hilbert module $H_A$. It is a celebrated theorem due to Kasparov that any countably generated Hilbert $A$-module $H$ is isometrically isomorphic to a complement of $H_A$; see \cite{Kasparov80}. From this point of view, countably generated Hilbert modules play a role akin to countably generated projective modules for C*-algebras. In fact, an algebraically finitely generated Hilbert module is a finitely generated projective module; see, e.g. \cite[Theorem 3.12]{APT2011}. Further, the monoid of isomorphism classes of finitely generated Hilbert $A$-modules is isomorphic to the monoid $\V(A)$ of the C*-algebra, as shown in \cite[Proposition 3.10]{APT2011}. We also remark that, as proved recently by Brown and Lin in \cite{BroLin23arX}, over a separable C*-algebra  every countably generated Hilbert module is projective (with bounded module maps as morphisms). This is a step forward in the direction of characterizing projective Hilbert modules over a C*-algebra.

An equivalence relation among countably generated Hilbert $A$-modules, weaker than isomorphism, was studied in \cite{Coward2008}. In there, the authors proved that the monoid arising from said equivalence relation may be identified with the complete Cuntz semigroup invariant $\Cu(A)$ of the C*-algebra  $A$.  (The terminology `semigroup' is used for historical reasons, although in fact $\Cu(A)$ is a monoid.)  The original (uncomplete) Cuntz semigroup $\W_\mathrm{C}(A)$ was constructed by Cuntz in \cite{Cun78DimFct} using positive elements and a suitable comparison relation among them that, when restricted to idempotents, yields  the usual Murray-von Neumann comparison. In short, we say that $a$ is Cuntz subequivalent to $b$, and write $a\precsim_\mathrm{Cu}b$, provided $a$ can be approximated arbitrarily well by elements of the form $xby$.  Compared with the construction of the group $\mathrm{K}_0$, this approach is advantageous since every C*-algebra has an abundance of positive elements but may have a complete lack of idempotents.  The exact relation between $\W_\mathrm{C}(A)$ and $\Cu(A)$ may be expressed by the isomorphism $\Cu(A)\cong\W_\mathrm{C}(A\otimes\mathcal{K})$, where $\mathcal{K}$ is the algebra of compact operators on an infinite dimensional Hilbert space. Alternatively, $\Cu(A)$  may be thought of as the completion of $\W_\mathrm{C}(A)$; see \cite[Theorem 3.2.8]{APT-Memoirs2018}. It was shown in \cite{Coward2008} that $\Cu(A)$ sits in a well-behaved category of partially ordered monoids, termed $\Cu$, in which each object admits suprema of increasing sequences, among other continuity properties. Furthermore, the assignment $A\mapsto\Cu(A)$ is a continuous  functor; see \cite{Coward2008, APT-Memoirs2018}.

The Cuntz semigroup plays a prominent role in the classification programme of C*-algebras initiated by G. A. Elliott and is a key ingredient in delimiting the optimal class of such algebras amenable to classification by the Elliott invariant (that consists essentially of $\mathrm{K}_0$, $\mathrm{K}_1$, and the trace simplex). Indeed, the examples constructed by A. S. Toms in \cite{Tom08ClassificationNuclear} can be distinguished by their Cuntz semigroups, but not by a swath of other well known topological invariants for C*-algebras that include, among others, the Elliott invariant and the stable rank.

When trying to adapt the ideas above to the purely algebraic setting one has to bear in mind that, in nature, the Cuntz semigroup is an analytic object. Thus one first needs to use an algebraic analogue of Cuntz comparison for general elements in a ring. We take advantage of the approach carried out in \cite{AGPS2010}, in which one defines $x\precsim_1y$, provided $x=rys$ for some elements $r,s\in R$, in order to construct a partially ordered monoid $\W(R)$ for any weakly s-unital ring $R$; see Paragraphs \ref{para:suni} and \ref{par:WR}.  By considering suitable equivalence classes of increasing sequences with respect to the relation $\precsim_1$, we obtain a monoid $\RLambda(R)$ in the category $\Cu$ that contains $\W(R)$. The object $\RLambda(R)$ can be conveniently identified with the monoid of intervals in $\W(R)$, but it is in general too large for our purposes. This differs fundamentally from what happens in the C*-algebraic case, and the reason may be found in the lack of a topology in $R$.  To remedy this drawback, we restrict our attention to a well-behaved partially ordered submonoid $\SR(R)$ of $\RLambda(R)$ which, for a C*-algebra $A$, is resemblant to $\Cu(A)$ and its role as the completion of $\W_\mathrm{C}(A)$. We will term $\SR(R)$ the \emph{Cuntz semigroup of the ring} $R$ and this is the main object of study in this paper. At this point, we mention that the construction of the Cuntz semigroup $\W_\mathrm{C}(A)$ for a C*-algebra $A$ served as inspiration to Hung and Li to introduce in \cite{HunLi21} a semigroup for any unital ring $R$, termed the \emph{Malcolmson semigroup}, and denoted by $\W_\mathrm{M}(R)$, in order to study Sylvester rank functions over the ring $R$.  To construct this semigroup one uses a relation stronger than $\precsim_1$ and, as it turns out, for any C*-algebra $A$ the semigroup $\W_\mathrm{C}(A)$ is a homomorphic image of $\W_\mathrm{M}(A)$ via an order-preserving map; see \cite[Lemma 5.1, Proposition 5.2]{HunLi21}. We shall review this construction and its relation to our semigroup $\W(R)$ in \autoref{sec:Malc}.

To explain how one constructs the Cuntz semigroup $\SR(R)$ we take a slight detour that finally yields two equivalent pictures of the same object and spurs our motivation at the same time. More concretely, given countably generated projective modules $P$ and $Q$  over $R$, we combine the approach carried out for Hilbert modules in \cite{Coward2008} with an abstraction of the above-mentioned relation $\precsim_1$ to write $P\precsim Q$ if the inclusion of any finitely generated module $X$ of $P$ may be facto\-rized through $Q$; see \autoref{def:PprecsimQ}. By antisymmetrizating the above relation, we get an ordered abelian monoid $\CP(R)$ and a natural surjective morphism $\V^*(R)\to \CP(R)$. As we show in \autoref{sec:CompEl}, if $R$ is either unit-regular or unital and semilocal, this is an isomorphism of abelian monoids, but not of ordered monoids, as $\V^*(R)$ is algebraically ordered, but $\CP(R)$ is not, except in trivial situations. Further investigation on countably generated projective modules structure leads us to reformulate the proof, obtained in \cite{Puninski2007}, that any such module can be written as a sequential inductive limit of free modules such that, for each $n$, the $n$th transition map consists of multiplication by a matrix $x_n$ with the property that $x_n=y_{n+1}x_{n+1}x_n$ for a suitable matrix $y_{n+1}$ (hence in particular $x_n\precsim_1 x_{n+1}$). Our arguments uncover additional and crucial information in such an inductive limit decomposition and in doing so we are able to relate the monoid $\CP(R)$ to the submonoid $\SR(R)$ of $\RLambda(R)$ consisting of (suitable equivalence classes of) increasing sequences $(x_n)$ with respect to $\precsim_1$ arising from inductive limits as above:

\begin{thmintro}[\ref{thm:nonunitalSRCP}, \ref{W(R)_closed}]
	\label{thmA}
	Let $R$ be any ring. Then $\CP(R)$ and the Cuntz semigroup  $\SR(R)$ of $R$ are order-isomorphic monoids.  Moreover, every increasing sequence in $\CP(R)$ (or $\SR(R)$) has a supremum.
\end{thmintro}

Despite the analogy of the construction of the Cuntz semigroup $\SR(R)$ of a ring $R$ with that of a C*-algebra, it is  unclear whether $\SR(R)$ is a $\Cu$-semigroup. We remedy this fact by considering the pair $\SCu(R)=(\RLambda(R), \SR(R))$. This is an instance of an object in the category $\SCu$ of pairs $(\Lambda,S)$, where $\Lambda$ is a monoid in the category $\Cu$  and $S$ is a submonoid of $\Lambda$ closed under suprema of a certain type of sequences. The definition of this new abstract category balances the fact that $\SR(R)$ might not be an object in $\Cu$ with an ambient object which does belong to the category and is still intimately related to $\SR(R)$. More concretely, we prove:

\begin{thmintro}[\ref{cor:Scufunctor}] 
	\label{thmB}
	Let $\mathrm{Rings}^{ws}$ denote the category of weakly s-unital rings.
	Then, the assignment \[
	\begin{array}{cccc}
		\SCu\colon &{\rm Rings}^{ws}&\longrightarrow&\SCu\\
		&R&\mapsto&(\Lambda (R), \SR (R))
	\end{array}
	\] is functorial.
\end{thmintro}	

In a subsequent paper (\cite{AntAraBosPerVil23arX:ContandIde}) we examine other structural properties of the object $\SCu(R)$, such as a natural notion of ideal and quotient, and how these notions parametrize the ideal lattice of a ring. There, we also show that the category $\SCu$ admits inductive limits and analyse when the assignment $R\mapsto \SCu(R)$ is continuous.

\medskip

We analyse the construction of this new Cuntz semigroup in a variety of situations. Firstly, since the original motivation of this paper comes from C*-algebra theory, we relate $\Cu(A)$ to $\SR(A)$ for any C*-algebra $A$, by showing the former is a retract of the latter, as follows:

\begin{thmintro}[\ref{thm:CuCs}]
	\label{thmC}
	Given a C*-algebra $A$, there exist ordered monoid morphisms $\varphi\colon \Cu (A) \to \SR (A)$ and $\phi\colon \SR (A)\to \Cu (A)$ that preserve suprema of increasing sequences and such that $\phi\circ\varphi=\mathrm{id}_{\Cu(A)}$.
\end{thmintro}
Secondly, we show that in a number of interesting examples outside the class of C*-algebras the monoids $\W(R)$ and $\SR(R)$, together with their order structure, can be identified:

\begin{thmintro}[\ref{prp:CuRegular}, \ref{prp:OrderSemi}, \ref{thm:computeSJ}]
	\label{thmD}
	Let $R$ be a unital ring, and let $P$ and $Q$ be countably generated projective $R$-modules. Then:
	\begin{enumerate}[{\rm (i)}]
		\item If $R$ is unit-regular, we have $[P]\leq [Q]$ in $\CP(R)$ precisely when $P$ is isomorphic to a submodule of $Q$. It follows that $\W(R)\cong\V(R)$ and $\SR(R)\cong\CP(R)\cong \RLambda(R)$. Thus $\SR(R)$ is a $\Cu$-semigroup.
		\item If $R$ semilocal, we have $[P]\leq [Q]$ in $\CP(R)$ precisely when $P$ is isomorphic to a pure submodule of $Q$.  In this case, as abelian monoids, we have $\SR(R)\cong\CP(R)\cong\V^*(R)$.
		\item	If $R$ is a nearly simple domain, then $\W(R)\cong\N\times\N$, and $(r,s)\leq (r',s')$ precisely when $r\leq r'$ and $r+s\leq r'+s'$. Moreover, $\SCu (J(R))\cong (\ol{\N} , 0)$.
	\end{enumerate}		
\end{thmintro}	
The article is organized as follows. In \autoref{sec:WR} we review the definition of the Cuntz semigroup $\W_\mathrm{C}(A)$ for a C*-algebra $A$ and the category its completion naturally belongs to, and we define its algebraic counterpart $\W(R)$ together with the natural construction $\RLambda(R)$, which will conveniently serve as an ambient monoid later on. In \autoref{sec:Malc} we relate the semigroup $\W(R)$ with the Malcolmson semigroup introduced in \cite{HunLi21}, and show both semigroups may be identified for unital von Neumann regular rings. \autoref{sec:SRCPR} constitutes the heart of this paper, where we construct the Cuntz semigroup $\SR(R)$ for a ring $R$ and prove Theorem \ref{thmA}. This is technically demanding as we need to split the proof into the unital and non-unital case. In \autoref{sec:CatSCu} we introduce the category $\SCu$ and establish Theorem \ref{thmB}. In \autoref{sec:CompEl} we study compact elements in $\SR(R)$ and prove parts (i) and (ii) of Theorem \ref{thmD}. We revisit C*-algebras in \autoref{sec_Cs} to relate $\SR(A)$ and $\Cu(A)$, thus proving Theorem \ref{thmC}. \autoref{sec:nearly} is devoted to the analysis of the class of nearly simple domains and to prove part (iii) of Theorem \ref{thmD}.
\section*{Acknowledgements}
We wish to thank Guillem Quingles for a thourough reading of a preliminary version of this manuscript and for making suggestions that helped to improve the exposition.


\section{The Cuntz semigroup of a \texorpdfstring{\Cs{}}{C*-algebra} and the semigroup \texorpdfstring{$\W(R)$}{W(R)}}
\label{sec:WR}

In this section we recall the definition of the Cuntz semigroup of a \Cs{} and its most natural adaptation to a purely algebraic framework. 

\begin{parag}[Diagonal sum in $M_\infty (R)$]\label{parag:FinMat}
	Given a ring $R$, we denote by $M_\infty (R)$ the ring of infinite matrices with only a finite number of nonzero entries. That is, given an element $x\in M_{\infty}(R)$ there exist $n,m\geq 1$ and a finite matrix $z\in M_{n\times m}(R)$ such that
	\[
	x= \begin{pmatrix}
		z & 0 & \cdots\\
		0 & 0 & \\
		\vdots & & \ddots
	\end{pmatrix}.
	\]
	
	We will call $z$ a \emph{finite representative} of $x$, and we will say that $x$ is the infinite matrix represented by $z$. We will tacitly identify $x$ and $z$ when no confusion can arise.
	
	Given two finite rectangular matrices $x\in M_{n_1\times m_1}(R)$ and $y\in M_{n_2\times m_2}(R)$, we will denote by $x\oplus y$ the infinite matrix
	\[
	\begin{pmatrix}
		x & 0 & 0 &  \cdots\\
		0 & y & 0 & \\
		0 & 0 & 0 & \\
		\vdots & & & \ddots
	\end{pmatrix}. 
	\]
	
	In other words, $x\oplus y$ is the infinite matrix represented by the rectangular matrix 
	\[
	\begin{pmatrix}
		x & 0\\
		0 & y 
	\end{pmatrix}\in M_{(n_1+n_2)\times (m_1+m_2)}(R). 
	\]
\end{parag}

\begin{parag}[Cuntz subequivalence and the Cuntz semigroup]\label{para:CuntzSgpDef}
	We shall denote by $\mathcal{K}$ the algebra of compact operators over an infinite-dimensional Hilbert space. Let $A$ be a \Cs{}. Given positive elements $a,b\in A$, we say that $a$ is \emph{Cuntz subequivalent to} $b$, in symbols $a\precsim_\mathrm{Cu} b$, provided that there is a sequence $(x_n)$ in $A$ such that $a=\lim_n x_nbx_n^*$. Equivalently, there are sequences $(x_n)$, $(y_n)$ in $A$ such that $a=\lim_n x_nby_n$ (see \cite{Cun78DimFct}).   We say that $a$ and $b$ are \emph{Cuntz equivalent}, in symbols $a\sim_\mathrm{Cu} b$, if both $a\precsim_\mathrm{Cu} b$ and $b\precsim_\mathrm{Cu} a$ occur. One can use the second equivalent definition of Cuntz subequivalence to extend $\sim_\mathrm{Cu}$ to arbitrary elements. That does not have any effect on the theory since, as it happens, $a^*a\sim_\mathrm{Cu}aa^*\in A_+$ for any $a\in A$.
	
	By extending this relation in the natural way to $M_\infty(A)_+$, one can define a partially ordered set 
	\[
	\W_\mathrm{C}(A)=M_\infty(A)_+/\!\!\sim,
	\]
	with order given by $[a]_\mathrm{Cu}\leq [b]_\mathrm{Cu}$ whenever $a\precsim_\mathrm{Cu} b$ (and where $[a]_\mathrm{Cu}$ denotes the equivalence class of $a$). It becomes a positively ordered semigroup by defining $[a]_\mathrm{Cu}+[b]_\mathrm{Cu}=[ a \oplus b ]_\mathrm{Cu}$. The semigroup $\W_\mathrm{C}(A)$, originally defined in \cite{Cun78DimFct}, is most currently referred to as the \emph{classical Cuntz semigroup}. The \emph{complete Cuntz semigroup} of a \Cs{} $A$ is $\Cu(A)=\W_\mathrm{C}(A\otimes\mathcal{K})$. (See \cite{APT2011}, and \cite{GarPer2022} for survey articles on the Cuntz semigroup.)
\end{parag}
Coward, Elliott, and Ivanescu introduced in \cite{Coward2008} the category $\Cu$, which captures continuity properties of the semigroup $\Cu(A)$.

\begin{parag}[The category $\Cu$ and abstract $\Cu$-semigroups] \label{para:AbsCuSgp}
	Given a positively ordered monoid $S$, we write $x\ll y$ (and say that $x$ is {\it compactly contained} in $y$, or that $x$ is \emph{way-below} $y$), if whenever $(y_n)$ is an increasing  sequence in $S$ for which the supremum $\sup _n y_n$ exists, then $y\le \sup _n y_n$ implies that there exists $k$ such that $x\le y_k$. (See \cite[I-1]{GieHof+03Domains}.)
	
	Using it, we consider the following axioms for $S$:
	
	\begin{enumerate}
		\item[(O1)] Every increasing sequence $(x_n)$ in $S$ has a supremum $\sup _n x_n\in S$.
		\item[(O2)] Every element $x\in S$ is the supremum of a sequence $(x_n)$ such that $x_n \ll x_{n+1}$ for all $n$. We say that $(x_n)$ is a {\it rapidly increasing sequence}.
		\item[(O3)] If $x',x,y',y\in S$ satisfy $x'\ll x$ and $y'\ll y$ then $x'+y'\ll x+y$.
		\item[(O4)] If $(x_n)$ and $(y_n)$ are increasing sequences in $S$, then $\sup _n (x_n+y_n)= \sup_n x_n + \sup_ny_n$.  
	\end{enumerate}
	
	An \emph{abstract Cuntz semigroup} (or just a \emph{$\Cu$-semigroup}) is a positively ordered monoid satisfying axioms (O1)-(O4). A $\Cu$-morphism between two $\Cu$-semigroups $S$ and $T$ is a positively ordered monoid morphism $f\colon S \to T$ that preserves compact containment and suprema of increasing sequences.
	The category $\Cu$ has as objects the $\Cu$-semigroups and as morphisms the $\Cu$-morphisms. It was shown in \cite{Coward2008} that the natural models of $\Cu$-semigroups are the complete Cuntz semigroups of \Cs{s}.
	
	Also, the natural models of $\Cu$-morphisms are the *-homomorphisms of \Cs{s}. More specifically, given \Cs{s} $A$ and $B$ and a *-homomorphism $\varphi\colon A\to B$, we may define $\Cu(\varphi)\colon\Cu(A)\to \Cu(B)$ by $\Cu(\varphi)([a])=[(\varphi\otimes\mathrm{id})(a)]$, for any $a\in (A\otimes\mathcal{K})_+$, which is a $\Cu$-morphism. In this way, the assignment $A\mapsto \Cu(A)$ determines a functor from the category of \Cs{s} to the category $\Cu$, which turns out to be continuous (see \cite{Coward2008} and also \cite{APT-Memoirs2018}).
	It was shown in \cite[Theorem 3.28]{APT-Memoirs2018} that $\Cu(A)$ is, suitably interpreted, a completion of $\W_\mathrm{C}(A)$.
	
	Elements that will become relevant in the theory are the so-called \emph{compact} elements. By definition, an element $x$ in a $\Cu$-semigroup $S$ is termed compact provided $x\ll x$. The natural sources of compact elements in Cuntz semigroups of \Cs{s} are the classes of projections, i.e. self-adjoint idempotents. In significant cases, these are the only ones (see, e.g. \cite{BroCiu09}). A Cuntz semigroup $S$ is said to be \emph{algebraic} provided every element is the supremum of a sequence of compact elements (\cite[Definition 5.5.1]{APT-Memoirs2018}).
	
	A \emph{sub-$\Cu$-semigroup} of a $\Cu$-semigroup $T$ is a submonoid $S$ of $T$ such that the inclusion $\iota \colon S \to T$ is a $\Cu$-morphism; see, for example, \cite{APT-Memoirs2018}. For example, let $\N$ be the set of natural numbers with $0$. Then, $\ol{\N}:= \N \cup \{\infty\}$ is a $\Cu$-semigroup and a submonoid of the $\Cu$-semigroup $[0,\infty]$, but it is not a $\Cu$-subsemigroup of $[0,\infty]$, because, for instance, $2 \ll 2 $ in $\ol{\N}$ but $2$ is not compactly contained in itself in $[0,\infty]$. 
\end{parag}
We now introduce an algebraic analog of the classical Cuntz semigroup. For this, we first need to consider a class of rings suitable to our needs.  The relation $\precsim_1$ below was already considered in \cite{AGPS2010}.

\begin{parag}[$s$-unital rings]\label{para:suni}
	We recall that a ring $R$ is said to be \emph{$s$-unital} if, for every element $a\in R$, there is $b\in R$ such that $a=ba=ab$. Evidently this includes all unital rings, $\sigma$-unital rings, and rings with local units.
	
	We will say that a ring $R$ is  \emph{weakly $s$-unital} if for every $n\geq 1$ and every element $a\in M_n(R)$, there are $b,c \in M_n(R)$ such that $a=bac$.
	
	By \cite[Lemma 2.2]{AAlgCol2004}, given a finite family $a_1,\dots, a_n$ of elements of an $s$-unital ring $R$, there is $b\in R$ such that $ba_i=a_i = a_ib$ for $i=1,\dots, n$. From this, one can show that if $R$ is $s$-unital then so is the ring $M_{\infty}(R)$. It also follows that any $s$-unital ring is weakly $s$-unital.
\end{parag}

\begin{parag}[The semigroup $\W(R)$]
	\label{par:WR}
	Let $R$ be any ring. Given elements $a,b\in R$, we write $a\precsim_1 b$ if there exist elements $r,t\in R$ such that
	\[
	a=rbt.
	\]
	The relation $\precsim_1$ is clearly transitive by construction. Assume further that $R$ is weakly $s$-unital, and then $\precsim_1$ is also reflexive. We write $a\sim_1 b$ provided $a\precsim_1 b$ and $b\precsim_1 a$.
	
	If $e,f$ are idempotents in $R$, then an easy argument shows that $e\precsim_1 f$ if and only if $e\sim f'$ and $f'\leq f$ in the sense that $e=xy$ whilst $yx=f'$, for elements $x\in eRf', y\in f'Re$. That is, the relation $\precsim_1$ agrees with the usual Murray-von Neumann subequivalence $\precsim_\mathrm{MvN}$ for idempotents. Therefore, if $e\sim f$, then $e\sim_1 f$, but the converse does not necessarily hold -- it will if all idempotents are \emph{finite}, in the sense that they do not contain proper equivalent copies of themselves.
	
	In case $A$ is a \Cs{} and $p,q\in A$ are projections, then it is known that $p\precsim_\mathrm{Cu} q$ if and only if $p=vv^*$ and $v^*v\leq q$. In other words, Cuntz subequivalence, when restricted to projections, agrees with the usual Murray-von Neumann subequivalence. It follows from this that $p\precsim_\mathrm{Cu} q$ precisely when $p\precsim_1 q$. However, this will not hold for general positive elements, and thus one cannot expect that our algebraic construction below coincides with the \Cs{ic} one. Notice also that \Cs{s} are in general neither weakly $s$-unital nor $s$-unital.
	
	We now extend the relation $\precsim_1$ to $M_\infty(R)$ and define
	\[
	\W(R)=M_\infty(R)/\!\!\sim_1.
	\]
	Denote the class of $a\in M_\infty(R)$ by $[a]$. As we show below, this partially ordered set becomes an abelian semigroup by defining $[a]+[b]=[(\begin{smallmatrix} a & 0 \\ 0 & b \end{smallmatrix})]$, for any $a,b\in M_\infty(R)$.
\end{parag}
\begin{lemma}\label{lma:WRPoM}
	For any weakly $s$-unital ring $R$, the poset $\W (R)$, equipped with the addition defined above, is a positively ordered commutative monoid.
\end{lemma}
\begin{proof}
	We first have to show that addition is well-defined.
	
	Note the following fact. Let $u\in M_{k\times l}(R)$ and $v\in M_{t\times s}(R)$ be finite matrices with coefficients in $R$, and let $x,y$ be the infinite matrices represented by $u$ and $v$ respectively. Then $x\precsim_1 y$ in $M_{\infty}(R)$ if and only if there are matrices $a\in M_{k\times t}(R)$ and $b\in M_{s\times l}(R)$ such that $u=avb$.
	
	For the addition, let $w,w'\in \W (R)$ and suppose that $u$ and $v$ are finite representatives of $w$, and that $u'$ and $v'$ are finite representatives of $w'$. Using the above observation we find finite matrices $a,b,a',b'$ of suitable sizes such that
	\[
	u = avb \text{ and }  u'=a'v'b'.
	\]
	We then have that $u\oplus u' = (a\oplus a')(v\oplus v')(b\oplus b')$. This shows that $(u\oplus u')\precsim_1 (v\oplus v')$, and similarly we have that
	$(v\oplus v')\precsim_1 (u\oplus u')$, and thus $[(u\oplus u')] = [ (v\oplus v')]$.
	
	The same argument shows that addition is compatible with the order in $\W (R)$. Further, it is clear that the class $[0]$ is the zero element and that addition is associative.
	
	To see that it is also commutative, let $w,w'\in \W (R)$, and let $u,u'$ be finite representatives of $w, w'$, respectively. Since $R$ is weakly $s$-unital, we may choose finite matrices $v,z$, $v',z'$ of suitable sizes such that $vuz=u$ and $v'u'z'=u'$.
	
	Then, we have
	\[
	\begin{pmatrix}
		u' & 0\\
		0 & u
	\end{pmatrix}=
	\begin{pmatrix}
		0 & v'\\
		v & 0
	\end{pmatrix}
	\begin{pmatrix}
		u & 0\\
		0 & u'
	\end{pmatrix}
	\begin{pmatrix}
		0 & z\\
		z' & 0
	\end{pmatrix}\precsim_1
	\begin{pmatrix}
		u & 0\\
		0 & u'
	\end{pmatrix}.
	\]
	Hence $u'\oplus u \precsim_1 u\oplus u'$.
	Thus, one gets
	\[
	w'+w =
	[u'\oplus u]\leq [u\oplus u']= w+w',
	\]
	and by symmetry $w+w'\leq w'+w$, showing that $w'+w=w+w'$, as desired.
\end{proof}
\begin{parag}[The semigroup $\V(R)$]
	We shall denote as customary by $\V(R)$ the semigroup of Murray-von Neumann equivalence classes of idempotents in $M_\infty(R)$, and we denote the class of an idempotent $e\in M_\infty(R)$ by $[e]_\mathrm{MvN}$. Our observations above mean that there is a an order-embedding $\iota:\V(R)\to \W(R)$, given by $[e]_\mathrm{MvN}\mapsto [e]$. This map is injective if $R$ is \emph{stably finite}, in the sense that $x+y=x$ in $\V(R)$ precisely when $y=0$.
	
	In particular, the next result shows how the different orders behave via $\iota$. Indeed, it is shown that every element of $\iota (V(R))$ can be complemented in $\W (R)$. It is worth noticing the converse does not always hold; see \autoref{rmk:ComplNonComp}.
	
	\begin{lemma}\label{lem:alg-order} 
		Let $R$ be a weakly $s$-unital ring, and let $x\in \iota (V(R))$ and $y\in \W (R)$. If $x\le y$, then there exists $z\in \W (R)$ such that $x+z=y$.
	\end{lemma}
	\begin{proof} 
		Let $e$ be an idempotent in $M_\infty (R)$ such that $x=[e]$, and let $v\in M_\infty (R)$ satisfy $y=[v]$. Using that $e\precsim_1 v$, we can find elements $r,s$ such that $e=rvs$. Since $e$ is idempotent, we may also assume that $r=er$ and $s=se$. Thus, the element $f:=vsr= vser$ is an idempotent in $vR$ satisfying $[f]=[e]=x$.
		
		Now set $w:=v-fv\in M_\infty (R)$. Since $R$ is weakly $s$-unital, there exist $a,b\in M_\infty (R)$ such that
		$awb=w$. Thus, we have
		\[
		v= 
		\begin{pmatrix} f & a\end{pmatrix} \begin{pmatrix} f & 0\\ 0 & v-fv\end{pmatrix}
		\begin{pmatrix} v \\ b\end{pmatrix}
		\]
		and, consequently, $y=[v]\le [f] + [w] =x+[w]$.
		
		Using once again that $R$ is weakly $s$-unital, let $c,d\in M_\infty (R)$ satisfy $v=cvd$. One gets
		\[
		\begin{pmatrix} f & 0  \\ 0 & v-fv\end{pmatrix}
		=
		\begin{pmatrix} fc \\ c-fc\end{pmatrix} v \begin{pmatrix} dsr & d-dsrv \end{pmatrix}.
		\]
		
		This shows that $[f]+[w]\le [v]=y$. Setting $z:=[w]$, we obtain 
		\[
		x+z=[f]+[w]=y,
		\]
		as desired.
	\end{proof}
\end{parag}
In a more concrete settting, recall that an element $a$ in a ring $R$ is said to be a \emph{von Neumann regular element} provided there is $x\in R$ such that $a=axa$. Upon replacing $x$ by $x'=xax$, we may also assume that $x=xax$. The ring $R$ is said to be a von Neumann regular ring if every element is von Neumann regular. (See \cite{vnrr}.)

\begin{lemma}
	\label{lma:Weasycase}
	Let $R$ be a stably finite von Neumann regular ring. Then, the natural map $\V(R)\to\W(R)$ is an order-isomorphism.
\end{lemma}
\begin{proof}
	Let $a\in R$, and write $a=axa$ with $x=xax$. It is then an easy exercise to verify that $a\sim_1 ax=:e$, which is an idempotent. Since matrices over a von Neumann regular ring are also von Neumann regular, this shows that the map $\V(R)\to\W(R)$ is surjective.
\end{proof}

The relation between $\precsim_1$ and $\precsim_\mathrm{Cu}$ for general positive elements in a C*-algebra is examined with some more detail in the lemma below.

\begin{lemma}
	\label{Cucls_Rmk}
	Let $A$ be a \Cs, and let $a,b\in A$. If $a\precsim_1 b$, then $a^*a\precsim_\mathrm{Cu}b^*b$. Therefore, there is a positively ordered monoid morphism $\iota_\mathrm{C}\colon \W(A)\to \W_\mathrm{C}(A)$, given by $\iota_\mathrm{C}([a])=[a^*a]_\mathrm{Cu}$.
\end{lemma}
\begin{proof}
	Suppose that $a\precsim_1 b$, and write $a=sbr$ for some $s,r\in A$. Then
	\begin{equation*}
		a^{*}a=r^{*}b^{*}s^{*}sbr\precsim_\mathrm{Cu} b^{*}s^{*}sb \le  \| s \|^2 b^{*}b\precsim_\mathrm{Cu} b^{*}b.
	\end{equation*}
	Notice that, in particular, since $x\sim_\mathrm{Cu} x^{2}$ for any $x\in A_{+}$, we have $a\precsim_\mathrm{Cu} b$ whenever $a,b\in A_{+}$ and $a\precsim_{1}b$.
\end{proof}

\begin{parag}[The semigroup $\RLambda(R)$]
	\label{pgr:lambdaseq}
	
	We now proceed to construct an object in the category $\Cu$ for any weakly $s$-unital ring $R$. In the C$^*$-setting this can be done by simply considering $\W_C(A\otimes \mathcal K)$, but there is no algebraic analogue of the compact operators $\mathcal K$. Another approach is through a completion of $\W(R)$ by a so called \emph{auxiliary relation} in $\W(R)$; see \cite[Definition I-1.11]{GieHof+03Domains},  \cite[Definition 2.2.4]{APT-Memoirs2018} and also  \autoref{rmk:auxiliary_Lamdaseq} below. But again, there is no clear algebraic analogue of such an auxiliary relation for our needs. 
	
	The approach below is partly inspired by the description of $\Cu(A)$ in \cite{Coward2008} using 
	certain increasing sequences. We will first make a general construction that works for a general ring, and then specialize to the weakly $s$-unital setting.
	
	
	Let $R$ be a ring. Set
	\[
	T(R)=\{(a_n)_n\mid a_n\in M_\infty(R) \text{ and } a_n\precsim_1 a_{n+1}\text{ for all }n \}.
	\]
	Given $(a_n), (b_n)\in T(R)$, write $(a_n)\precsim (b_n)$ if for any $n$, there exists $m$ with $a_n\precsim_1 b_m$. We also write $(a_n)\sim (b_n)$ provided $(a_n)\precsim (b_n)$ and $(b_n)\precsim (a_n)$. Notice that, since for $(a_n)\in T(R)$ we have $a_n\precsim_1a_{n+1}$, we see that $(a_n)\precsim (a_n)$ and thus the relation $\precsim$ is reflexive even if $R$ is not unital. It is also clearly transitive.
	Let
	\[\RLambda(R)=T(R)/{\sim},
	\]
	which is a partially ordered set with the order induced by $\precsim$. We denote by $[(a_n)]$ the $\sim$-equivalence class of a sequence $(a_n)$ in $T(R)$. We now equip $\RLambda(R)$ with a semigroup structure, and to this end we need to be careful with the choices of representatives.
	
	Thus, in analogy to the terminology introduced in \autoref{parag:FinMat}, given an element $w\in \RLambda(R)$, a {\it finite matricial representative} of $w$ is any sequence $(u_n)$ such that $u_n\in M_{k_{n+1}\times k_n}(R)$, where $(k_n)$ is a sequence of positive integers, for which there exist matrices $v_{n+1}\in M_{k_{n+1}\times k_{n+2}}(R)$ and $z_{n+1}\in M_{k_{n+1}\times k_n}(R)$ such that $u_n= v_{n+1}u_{n+1}z_{n+1}$ for all $n$, and with $w=[(x_n)]$, where $x_n$ is the infinite matrix represented by $u_n$ for each $n\in \N$.
	
	For $w,w'\in \RLambda (R)$, let $(u_n)$ and $(u'_n)$ be finite matricial representatives of $w$ and $w'$ respectively. The sum $w+w'$ is then defined as
	\[
	w+w'= [(u_n\oplus u_n')]\in \RLambda (R).
	\]
\end{parag}

\begin{lemma}\label{lma:LambdaseqPoM}
	For any ring $R$, the poset $\RLambda (R)$, equipped with the addition defined above, is a positively ordered commutative monoid.
\end{lemma}

\begin{proof}
	The argument is similar to \autoref{lma:WRPoM}. We sketch the main steps in the proof.
	
	If $w,w'\in \RLambda (R)$, let $(u_n),(v_n)$ be two finite matricial representatives of $w$, and let $(u'_n)$ and $(v_n')$ be finite matricial representatives of $w'$. Given $n\in \N$, we see, using the first observation in the proof of \autoref{lma:WRPoM}, that there is $m\in \N$ and finite matrices $a,b,a',b'$ of suitable sizes such that
	\[
	u_n = av_mb, \qquad u_n'=a'v_m'b'.
	\]
	Then $u_n\oplus u_n' = (a\oplus a')(v_m\oplus v_m')(b\oplus b')$, and thus $(u_n\oplus u_n')\precsim (v_n\oplus v_n')$. Likewise, $(v_n\oplus v_n')\precsim (u_n\oplus u_n')$, whence $[(u_n\oplus u_n')] = [ (v_n\oplus v_n')]$.
	
	The rest of the argument follows the lines of \autoref{lma:WRPoM}.
\end{proof}
\begin{proposition}
	\label{rmk:ifneedbe}
	Let $R$ be any ring. Then every increasing sequence in $\RLambda(R)$ has a supremum. If, further, $R$ is weakly $s$-unital, then $\RLambda(R)$ is a $\Cu$-semigroup.
\end{proposition}
\begin{proof}
	This is fundamentally contained in the arguments of \cite[Proposition 3.1.6]{APT-Memoirs2018}. We offer some details for the convenience of the reader.
	
	Let $([x_k])_k$ in $\RLambda(R)$ be an increasing sequence. We thus have that $x_k\precsim x_{k+1}$ for all $k$. Write $x_k=(x_n^{(k)})_n$, with $x_n^{(k)}\precsim_1x_{n+1}^{(k)}$ for all $n$ and all $k$. By an inductive process, we find an increasing sequence $n_k$ such that $x_{n_i+j}^{(i)}\precsim_1 x_{n_k}^{(k)}$ if $i+j\leq k$.
	
	To see this, set $n_1=0$ and using that $x_1\precsim x_2$, find $n_2$ such that $x_{n_1+1}^{(1)}=x_1^{(1)}\precsim_1x_{n_2}^{(2)}$. If $n_i$ is constructed for $i\leq k$, we use that $x_1,\dots,x_k\precsim x_{k+1}$ to find $n_{k+1}$ such that $x_{n_i+k}^{(1)},x_{n_2+k-1}^{(2)},\dots,x_{n_k+1}^{(k)}\precsim_1x_{n_{k+1}}^{(k+1)}$, and thus the induction is complete.
	
	After reindexing we may assume that $n_i=i$ and therefore $x_{i+j}^{(i)}\precsim_1 x_{i+j}^{(i+j)}$ for all $i,j$. Setting $y_n=x_n^{(n)}$, we have that $y:=(y_n)$ satisfies $[y]=\sup_n[x_k]$. This shows the first part of the statement.
	
	Assume now that $R$ is weakly $s$-unital, hence $x\precsim_1x$ for each $x\in M_\infty(R)$. Then, if $[(x_n)]\in\RLambda(R)$, the sequence $x^{(k)}:=(x_1,\dots,x_k,x_k,\dots)$ belongs to $T(R)$, and it is not difficult to show that $[(x_n)]=\sup_k [x^{(k)}]$.
	
	Using this fact, one may check that, given $[(x_n)],[(y_n)]\in \RLambda(R)$, we have $[(x_n)]\ll [(y_n)]$ in $\RLambda(R)$ if and only if there is $m$ such that $x_n\precsim_1 y_m$ for all $n$.
	
	From this, one gets that axioms (O2)-(O4) are satisfied in $\RLambda(R)$ and, combined with the first part of the proof, we obtain that $\RLambda(R)$ is a $\Cu$-semigroup.
\end{proof}

In the weakly $s$-unital setting, the semigroup $\RLambda(R)$ can be conveniently identified with the monoid of intervals in the semigroup $\W(R)$. We make this connection explicit below, and in the sequel we will use both pictures interchangeably. Intervals have been used in many places, in connection with \Cs{s} and other algebraic structures; see, for example, \cite{Weh96}, \cite{Per97}, \cite{AntBosPer11}, or \cite{APT-Memoirs2018}. Our discussion below consists of well-known facts on intervals.
\begin{parag}[Intervals]
	\label{par:intervals}
	Let $M$ be a positively ordered monoid. Recall that an \emph{interval} in $M$ is a subset $I\subseteq M$ which is upward directed and downward hereditary. The set of intervals is denoted by $\Lambda(M)$, and it becomes a positively ordered monoid by defining
	\[
	I+J=\{z\in M\mid z\leq x+y\text{ where }x\in I \text{ and }y\in J\},
	\]
	and where order is given by set inclusion.
	
	We say that an interval $I$ in $M$ is \emph{countably generated} provided $I$ has a countable, cofinal subset. Equivalently, there is an increasing sequence $(x_n)$ in $I$ such that $I=\{x\in M\mid x\leq x_n\text{ for some }n\}$.
	The set of countably generated intervals in $M$ is denoted by $\Lambda_\sigma(M)$, which is clearly a positively ordered submonoid of $\Lambda (M)$. Indeed, if $I$ and $J$ have countable cofinal subsets $(x_n)$ and $(y_n)$, respectively, then $(x_n+y_n)$ is a countable cofinal subset for $I+J$.
	
	We have a positively ordered monoid morphism $\phi\colon M\to \Lambda_\sigma(M)$ given by $\phi(x)=[0,x]$, which is an order-embedding. Notice that every increasing sequence $(I_n)$ in $\Lambda_\sigma(M)$ has a supremum, simply given by $\cup_nI_n$. From this, and writing every interval $J\in \Lambda_\sigma(M)$ as $J=\cup_n[0,y_n]$ for an increasing cofinal sequence $(y_n)$, it follows easily that $I\ll J$ if and only if there is $y\in J$ such that $I\subseteq [0,y]$.
	
	It is then clear that, for each $I\in \Lambda_\sigma(M)$ and $x\in I$, we have $[0,x]\ll I$ and, if $(x_n)$ is an increasing sequence in $I$ which is a countable cofinal subset, then $I=\sup [0,x_n]$ with $[0,x_n]\ll [0,x_n]\ll [0,x_{n+1}]$.
	
	It is also easy to verify that addition in $\Lambda_\sigma(M)$ is compatible with suprema and the compact containment relation.
\end{parag}

\begin{lemma}
	Let $M$ be a positively ordered monoid. Then, its set of countably generated intervals $\Lambda_\sigma (M)$ is always an algebraic $\Cu$-semigroup. 
\end{lemma}
\begin{proof}
	This follows directly from the discussion carried out in  \autoref{par:intervals}.
\end{proof}

\begin{parag}[The semigroup $\WLambda(R)$]\label{dfn:IntR}
	Let $R$ be a weakly $s$-unital ring. We let
	\[
	\WLambda(R)=\Lambda_\sigma(\W(R)).
	\]
	Notice that, if for example $D$ is any division ring, one has $\W(D)\cong\mathbb{N}$, where the isomorphism is given by assigning to each matrix its rank. It follows that $\WLambda (D)\cong \overline{\N}$. 
	
	Another example is given by purely infinite simple rings. Recall that a unital simple ring $R$ is said to be \emph{purely infinite} provided $R$ is not a division ring and, for every non-zero element $a\in R$, there are elements $x,y\in R$ such that $xay=1$ (see \cite[Theorem 1.6]{AraGooPar02}) . This implies in particular that $a\sim_1 b$ for any non-zero elements $a,b\in R$, and thus $\W(R)\cong\{0,\infty\}$ and also $\WLambda(R)\cong\{0,\infty\}$.
\end{parag}

\begin{proposition}\label{lma:PictureSR}
	Let $R$ be a weakly $s$-unital ring. Then there is an ordered monoid isomorphism
	\[
	\WLambda(R)\cong \RLambda(R).
	\]
\end{proposition}
\begin{proof}
	Given $I\in\WLambda(R)$, let $([a_n])$ be an increasing sequence in $\W(R)$ such that it is a cofinal subset for $I$. Define $\varphi\colon \WLambda(R)\to\RLambda(R)$ by $\varphi(I)=[(a_n)]$.
	
	If $I\subseteq J$ and $([a_n])$, $([b_n])$ are increasing, countable cofinal subsets for $I$ and $J$, respectively, then for each $n$, there is $m$ such that $[a_n]\leq [b_m]$. Therefore $(a_n)\precsim (b_n)$ and thus $\varphi$ is well defined and order-preserving. It is easy to verify that $\varphi$ also preserves addition. Evidently, $\varphi$ is surjective.
	
	Let us check that $\varphi$ is an order-embedding. Suppose that $I,J\in\WLambda(R)$ satisfy $\varphi(I)\leq \varphi(J)$. Let $([a_n]), ([b_n])$ be increasing, countable cofinal sequences for $I$ and $J$, respectively. Then by definition of the order in $\RLambda(R)$ we have that, for each $n$, there is $m$ with $a_n\precsim_1 b_m$, which clearly implies that $I\subseteq J$.
	
	Therefore $\varphi$ is an ordered monoid isomorphism.
\end{proof}
\begin{remark}
	We note that \autoref{lma:PictureSR} offers an alternative proof that $\RLambda(R)$ is an object in $\Cu$ in the weakly $s$-unital setting.
\end{remark}

\section{The Malcomlson semigroup, Sylvester rank functions, and dimension functions}
\label{sec:Malc}

In this section, we briefly recall the construction of the Malcomlson semigroup as introduced in \cite{HunLi21} and its relation to $\W(R)$.

\begin{parag}[The Malcomlson semigroup]
	Let $R$ be a unital ring. Following \cite{HunLi21} with a slight change of notation, we define a relation $\precsim_0$  in $M_\infty(R)$ by $a\precsim_0 b$ if either $a\precsim_1 b$ or $a=(\begin{smallmatrix} c & 0 \\ 0 & d \end{smallmatrix})$ and $b=(\begin{smallmatrix} c & e \\ 0 & d \end{smallmatrix})$ for suitable $c,d,e\in M_\infty(R)$.
	
	Define $\precsim_\mathrm{M}$ to be the transitive closure of $\precsim_0$, so that $a\precsim_\mathrm{M} b$ if and only if there exist $a_1,\dots,a_n\in M_\infty(R)$ with $a=a_1\precsim_0 a_2\precsim_0\dots \precsim_0 a_n=b$. Set $\sim_\mathrm{M}$ as the antisymmetrization of $\precsim_\mathrm{M}$ and define
	\[
	\W_\mathrm{M}(R)=M_\infty(R)/\!\!\sim_\mathrm{M}.
	\]
	Denote the elements in $\W_\mathrm{M}(R)$ by $[a]_\mathrm{M}$. It follows that $\W_\mathrm{M}(R)$ is a positively ordered abelian semigroup with addition given by $[a]+[b]=[(\begin{smallmatrix} a & 0 \\ 0 & b \end{smallmatrix})]$ and order induced by $\precsim_\mathrm{M}$. 
	
	Since clearly $\precsim_\mathrm{M}$ is weaker than $\precsim_1$, we have a positively ordered monoid morphism $\iota_\mathrm{M}\colon \W(R)\to\W_\mathrm{M}(R)$, given by $[a]\mapsto [a]_\mathrm{M}$.
	
	As shown in \cite[Lemma 5.1]{HunLi21}, if $A$ is a C*-algebra and $a,b\in M_\infty(A)$, one has that $a\precsim_\mathrm{M}b$ implies $a\precsim_\mathrm{Cu}b$. There is then a positively ordered monoid morphism $\tilde{\iota}_\mathrm{C}\colon \W_{\mathrm{M}}(A)\to\W_\mathrm{C}(A)$. Combining this with \autoref{Cucls_Rmk}, we have the following commutative diagram
	\[
	\xymatrix{
		\W(A) \ar[r]^{\iota_\mathrm{M}} \ar[dr]_{\iota_\mathrm{C}}& \W_\mathrm{M}(A) \ar[d]^{\tilde{\iota}_\mathrm{C}}\\
		& \W_\mathrm{C}(A)
	}.
	\]
\end{parag}
\begin{parag}[States, Sylvester rank functions and dimension functions]
	\label{par:dim_func}
	Given a positively ordered monoid $S$ with an order-unit $u$, recall that a \emph{state} on $S$ normalized at $u$ is a positively ordered semigroup map $s\colon S\to [0,\infty)$ such that $s(u)=1$. The set of states is customarily denoted by $\mathrm{St}(S,u)$.
	
	Let $R$ be a unital ring. A map $d\colon M_\infty(R)\to [0,\infty)$ such that the following conditions hold:
	\begin{enumerate}[{\rm (i)}]
		\item $d(0)=0$ and $d(1)=1$;
		\item $d(ab)\leq d(a), d(b)$ or, equivalently, $d(a)\leq d(b)$ whenever $a\precsim_1 b$;
		\item $d(a)+d(b)=d(\begin{smallmatrix} a & 0 \\ 0 & b\end{smallmatrix})$,
	\end{enumerate}
	will be called a \emph{dimension function}. If, furthermore, $d$ satisfies
	\begin{enumerate}[{\rm (iv)}]
		\item  $d(a)+d(b)\leq d(\begin{smallmatrix} a & c\\ 0 & b\end{smallmatrix})$,
	\end{enumerate}
	then we say that $d$ is a \emph{Sylvester matrix rank function} for $R$.
	
	Denote the set of dimension functions by $\mathrm{DF}(R)$, and the subset of Sylvester matrix rank functions by $\mathbb{P}(R)$. Note that $\mathbb{P}(R)$ may be identified with the states on $\W_\mathrm{M}(R)$ normalized at $[1]$. Indeed, given $d\in \mathbb{P}(R)$, set $s_d([a])=d(a)$. Conversely, if $s\in\mathrm{St}(W_\mathrm{M}(R),[1])$, define $d_s(a)=s([a])$. Likewise, one may check that the set of dimension functions may be identified with $\mathrm{St}(\W(R),[1])$.
\end{parag}
\begin{lemma}
	\label{lma:CvsMclm}
	Let $R$ be an $s$-unital ring. Let $a,b,c\in M_\infty(R)$. If $a$ is a von Neumann regular element, then
	\[
	\begin{pmatrix} a & 0 \\ 0 & b\end{pmatrix}\precsim_1 \begin{pmatrix} a & c \\ 0 & b\end{pmatrix}.
	\]
	In particular, if $R$ is a unital von Neumann regular ring, then $\precsim_1$ is equivalent to $\precsim_{\mathrm{M}}$ and thus the natural map $\iota_{\mathrm{M}}\colon \W(R)\to \W_{\mathrm{M}}(R)$ is an order-isomorphism, whence $\mathrm{DF}(R)=\mathbb{P}(R)$.
\end{lemma}
\begin{proof}
	Let $S=M_\infty(R)$. Since $a$ is von Neumann regular and $R$ is $s$-unital, there are elements $a',b'\in S$ such that $a=aa'a$ and $b=b'b=bb'$ and $c=cb'=b'c$. Therefore
	\[
	\begin{pmatrix} aa' & 0\\ 0 & b'\end{pmatrix}\cdot \begin{pmatrix} a & c \\ 0 & b \end{pmatrix}\cdot \begin{pmatrix} a'a & -a'c \\ 0 & b'\end{pmatrix}=\begin{pmatrix} a & -aa'c+aa'cb' \\ 0 & b'bb'\end{pmatrix}=\begin{pmatrix} a & 0 \\ 0 & b\end{pmatrix}.
	\]
	The second part of the statement follows by definition of $\precsim_\mathrm{M}$: if $a\precsim_\mathrm{M}b$, then there are elements $a_1,\dots,a_n\in M_\infty(R)$ such that $a=a_1\precsim_0\dots\precsim_0 a_n=b$. For each $i$, we have by definition of $\precsim_0$ and the first part of the proof that $a_i\precsim_1 a_{i+1}$, whence $a\precsim_1 b$.
\end{proof}
\begin{parag}[Functionals]
	\label{par:functional} Let $S$ be a $\Cu$-semigroup. Recall that a \emph{functional} on $S$ is a positively ordered monoid morphism $\lambda\colon S\to [0,\infty]$ that respects suprema of increasing sequences. The set of functionals on $S$ is denoted by $F(S)$.
	
\end{parag}

\begin{proposition}
	\label{lma:functionals} Let $R$ be a unital ring. Then, any dimension function $d$ on $R$ induces a unique functional $\lambda_d \in F(\WLambda(R))$ such that $\lambda_d (\phi ([x])) = d(x)$ for all $x\in M_{\infty}(R)$, where $\phi\colon \W (R)\to \WLambda (R)$ is the canonical homomorphism. In particular, this is the case for any Sylvester matrix rank function.
\end{proposition}

\begin{proof}
	By \autoref{lma:PictureSR}, $\WLambda(R)$ is order-isomorphic to $\RLambda(R)$. Let $x\in \RLambda(R)$ and let $(x_n)$ be a representative for its class, which by definition satisfies $x_n\precsim_1x_{n+1}$ for each $n$. Given a dimension function $d\in DF(R)$, set $\lambda_d(x)=\sup_n d(x_n)$. By construction, if $(x_n)\precsim (y_n)$, then for each $n$ there is $m$ with $x_n\precsim_1 y_m$, and thus $\sup_n d(x_n)\leq \sup_md(y_m)$. This yields a well defined, order-preserving map $\lambda_d\colon \RLambda(R)\to [0,\infty]$.
	
	Since addition in $\RLambda(R)$ is defined componentwise, we see that $\lambda_d$ is additive. It remains to check that $\lambda_d$ preserves suprema of increasing sequences. To do so, we recall how suprema of increasing sequences are constructed in $\RLambda(R)$; see \autoref{rmk:ifneedbe}. Given a sequence $(x_k)$ in $\RLambda(R)$ such that $x_k\precsim x_{k+1}$ for all $k$, we write $x_k=(x_n^{(k)})_n$ and assume, reindexing if necessary, that $x_{i+j}^{(i)}\precsim_1 x_{i+j}^{(i+j)}$ for all $i,j$. Then $y=(y_n)$, where $y_n:=x_n^{(n)}$, is the supremum of the sequence $(x_k)_k$.
	
	Since $x_n^{(k)}=x_{n-k+k}^{(k)}\precsim_1 x_n^{(n)}$ whenever $n\geq k$, we have
	\[
	\sup_k\lambda_d(x_k)=\sup_k\sup_n d(x_n^{(k)})\leq \sup_n d(x_n^{(n)})=\lambda_d(y).
	\]
	Since the reverse inequality $\lambda _d(y)\le \sup_k\lambda_d(x_k)$ is obvious, this shows that $\lambda_d(y)=\sup_k\lambda_d(x_k)$.
	
	Uniqueness of $\lambda_d$ follows from the fact that $[(x_n)]= \sup _n \phi ([x_n])$ for $[(x_n)]\in \RLambda (R)$.
\end{proof}

\section{The Cuntz semigroup of a ring}\label{sec:SRCPR}

As motivated in the introduction, in this section  we pursue a construction of a semigroup using countably generated projective modules over a ring $R$. We first define a subsemigroup $\SR(R)$ of $\RLambda(R)$ consisting of certain classes of sequences that will yield a presentation of a countably generated module. Secondly, we will use a direct approach building another semigroup $\CP(R)$ as equivalence classes of countably generated projective right $R$-modules using a relation weaker than isomorphism and related to the one used in \cite{Coward2008}.

The main result of the section is that the semigroups $\SR(R)$ and $\CP(R)$ happen to be isomorphic in complete generality (see \autoref{thm:nonunitalSRCP}). The semigroup $\SR(R)$ is closed under suprema of increasing sequences, and it is the invariant that best resembles the Cuntz semigroup for rings. However, it is not obviously an object in the category $\Cu$. We remedy this by considering the pair $(\RLambda(R),\SR(R))$, of which the first component belongs to $\Cu$; see \autoref{sec:CatSCu}.

\begin{parag}[The semigroup $\SR(R)$]
	\label{subsec:WR}
	Let $R$ be a ring. Let $\preS (R)$ be the subset of $T(R)$ consisting of those sequences $(x_n)_n$, with $x_n\in M_{\infty}(R)$ for all $n\in\mathbb N$, such that for every $n$ there exists $y_{n+1}\in M_{\infty}(R)$ with $y_{n+1}x_{n+1}x_{n}=x_{n}$.
	
	We define $\SR(R)=\preS (R)/{\sim}$, which by construction is a subset of $\RLambda(R)$; see \autoref{pgr:lambdaseq}. We call $\SR(R)$ the \textit{Cuntz semigroup} of $R$.  Let us check it is a subsemigroup of $\RLambda(R)$. We continue to use the terminology introduced in \autoref{sec:WR}, and thus given an element $w\in \SR(R)$, a {\it finite matricial representative} of $w$ is any sequence $(u_n)$ such that $u_n\in M_{k_{n+1}\times k_n}(R)$, where $(k_n)$ is a sequence of positive integers, for which there exists $v_{n+1}\in M_{k_{n+1}\times k_{n+2}}(R)$ such that $u_n= v_{n+1}u_{n+1}u_n$ for all $n$, and with $w=[(x_n)]$, where $x_n$ is the infinite matrix represented by $u_n$ for each $n\in \N$. Therefore, for $w,w'\in \SR (R)$, we let $(u_n)$ and $(u'_n)$ be finite matricial representatives of $w, w'$, respectively, and define $w+w'= [(u_n\oplus u_n')]$.
	
	That addition is well defined, associative, commutative, and that the class of the zero sequence is the identity follows as in the arguments in Lemmas~\ref{lma:WRPoM} and \ref{lma:LambdaseqPoM}.
	
	The definition of $\SR(R)$ is admittedly one-sided, and one can define a left version of the semigroup $\SR _{l}(R)$ as the semigroup whose elements are equivalence classes under $\sim$ of sequences $(x_{n})_{n}$ in $M_{\infty}(R)$ such that
	\begin{equation*}
		x_{n}=x_{n}x_{n+1}y_{n+1}
	\end{equation*}
	for every $n\in\mathbb{N}$.
	
	It follows that $\SR _{l}(R)\cong \SR (R^{\text{op}})$. However, $\SR _l (R)$ is not always isomorphic to $\SR (R)$, as the example below testifies.
\end{parag}

\begin{example}\label{rmk:LeftWR}
	There exists a ring $R$ such that $\SR(R)\not\cong\SR_l(R)$.
\end{example}
\begin{proof}
	Let $K$ be a field and let $K[x_{0},x_{1},\ldots]$ be the free algebra on countably many variables subject to the (non-commutative) relations $x_{i+1}x_{i}=x_{i}$, and take the subring $R$ of polynomials with zero constant term.
	
	Given a monomial $p=\prod_{i\leq j\leq n}x_{j}^{t_{j}}$ with $t_{i},t_{n}>0$, we will write $\text{st}(p)=i$ and $\text{end}(p)=n$.
	
	Then, note that given any other monomial $q=x_{l}^{s_{l}}\ldots x_{m}^{s_{m}}$ with $s_{l},s_{m}>0$, we either have $pq=q$ if $\text{st}(p)>\text{st}(q)$ or $pq>p$ in the lexicographic order otherwise. Note that we always have $\text{st}(p)\geq \text{st}(pq)$.
	
	Thus, let $P=A_{1}p_{1}+\ldots + A_{n}p_{n}\in M_{\infty} (R)$ be a non-zero polynomial with $p_{1},\ldots ,p_{n}$ distinct monomials and $A_i\in M_{\infty}(K)\setminus \{0\}$ for all $i$. We may assume that $\text{st}(p_{1})\geq \text{st}(p_{i})$ for each $i$ and that $p_{1}$ is the smallest monomial in the lexicographic order among all monomials $p_{i}$ with $\text{st}(p_{1})=\text{st}(p_{i})$.
	
	Now let $Q=B_{1}q_{1}+\ldots + B_{m}q_{m}\in M_{\infty} (R)$ be another polynomial with $q_{1},\ldots ,q_{m}$ monomials and $B_j\in M_{\infty}(K)$. If $PQ=P$, $A_{1}p_{1}$ must be equal to a combination of the form $M_{1}p_{i_{1}}q_{j_{1}}+\ldots + M_{r}p_{i_{r}}q_{j_{r}}$, where $M_k\in M_{\infty}(K)$. In particular, $p_{i_{k}}q_{j_{k}}=p_{1}$ for each $k$.
	
	However, this would imply 
	$$ \text{st}(p_1)\ge \text{st}(p_{i_k})\ge \text{st}(p_{i_{k}}q_{j_k})=\text{st}(p_{1}),$$
	and so $\text{st}(p_{i_k})= \text{st}(p_1)$ for all $k$. This implies $p_1= p_{i_{k}}q_{j_{k}}>p_{i_{k}}\geq p_{1}$, a contradiction. Therefore, we must have $P=0$, which shows that $\SR _{l}(R)=0$.
	
	This is not the case for $\SR (R)$, since the sequence $(x_{0},x_{1},x_{2},\ldots)$ is in $\preS (R)$.
\end{proof}

The lemma below shows one of the main properties of $\SR (R)$: suprema exist for increasing sequences.

\begin{lemma}\label{W(R)_closed}
	Let $R$ be any ring. Then every increasing sequence in $\SR (R)$ has a supremum.
\end{lemma}
\begin{proof}
	Let $([x_k])_k$ be an increasing sequence in $\SR(R)$. Write $x_k=(x_n^{(k)})_n$ for each $k$, and find $y_{n+1}^{(k)}\in M_\infty(R)$ such that $x_n^{(k)}=y_{n+1}^{(k)}x_{n+1}^{(k)}x_n^{(k)}$.
	
	Note that $([x_k])_k$ also belongs to $\RLambda(R)$ and that the order in $\SR(R)$ and $\RLambda(R)$ is the same. By \autoref{rmk:ifneedbe}, the sequence $([x_k])_k$ has a supremum $z=[(z_n)]$ in $\RLambda(R)$. It is enough to check that $z\in\SR(R)$.
	
	To this end, recall from \autoref{rmk:ifneedbe} that, after a possible reindexing, one may assume that $x_{i+j}^{(i)}\precsim_1 x_{i+j}^{(i+j)}$ for all $i,j$ and then we take $z_n=x_n^{(n)}$. Since $x_n^{(n)}\precsim_1 x_{n+1}^{(n+1)}$, there are $a_{n+1}$, $b_{n+1}$ such that $x_n^{(n)}=a_{n+1}x_{n+1}^{(n+1)}b_{n+1}$. Thus
	\[
	x_n^{(n)}=y_{n+1}^{(n)}x_{n+1}^{(n)}x_n^{(n)}=y_{n+1}^{(n)}a_{n+1}x_{n+1}^{(n+1)}b_{n+1}x_n^{(n)}.
	\]
	If we let $u_n=x_n^{(n)}b_n$, it follows that $u_n=(y_{n+1}^{(n)}a_{n+1})u_{n+1}u_n$, whence $(u_n)\in \preS(R)$. By construction we have $z=[(z_n)]=[(u_n)]\in \SR(R)$, as was to be shown.
\end{proof}

We now proceed to introduce a semigroup steming directly from the class of countably generated projective $R$-modules, although with a new equivalence relation inspired by the construction in \cite{Coward2008}.

\begin{parag}[The semigroup $\CP(R)$ for a unital ring $R$]
	\label{def:PprecsimQ}
	Let $R$ be a unital ring. Let us denote by $\preCP (R)$ the class of all countably generated projective right $R$-modules. The first natural relation between (countably generated) projective modules is given by isomorphism. This yields the semigroup $\V^*(R)$ of isomorphism classes of countably generated projective modules, with addition given by direct sum. This semigroup has been successfully considered in \cite{HPCrelle} and \cite{Herbera2014}. Below, we weaken the above relation to another relation $\sim$ compatible with direct sum, thereby constructing an abelian semigroup $\CP(R)$. Thus, in particular, there is a natural surjective semigroup homomorphism $\Phi_R\colon\V^*(R)\to\CP(R)$, which will be an isomorphism in some cases, but not always; see \autoref{sec:CompEl}.
	
	Given $P,Q\in \preCP(R)$, we will write $P\precsim Q$ if and only if, for every finitely generated submodule $X$ of $P$, there exists a factorization of the inclusion of $X$ in $P$ by $Q$, that is, there are module morphisms $\phi\colon X\to Q$ and $\phi\colon Q\to P$ such that $\psi\circ\phi=\id_X$. Namely, the diagram below is commutative:
	
	\[
	\xymatrix{X\ar@{->}^{\phi}[r]\ar@/_1pc/[rr]_{\id_X} &Q\ar@{->}^{\psi}[r]&P}.
	\]
	
	We define the partially ordered set $\CP (R)$ to be
	\[
	\CP(R):=\preCP(R)/{\sim},
	\]
	where $\sim$ is the antisymmetrization of $\precsim$. Given an element $P$ in $\preCP (R)$, we will denote its equivalence class by $[P]$. For modules $P,Q\in \preCP(R)$, we define $[P]+[Q]=[P\oplus Q]$.
\end{parag}

\begin{lemma}
	The relation $\precsim$ is reflexive, transitive, and compatible with the direct sum of projective modules. Therefore, $\CP(R)$ is a commutative semigroup.
\end{lemma}
\begin{proof}
	That $\precsim$ is reflexive is trivial to verify. Let us show that it s transitive.
	%
	
	To this end, let $P_1$, $P_2$, $P_3\in\preCP(R)$ and assume that $P_1\precsim P_2$ and $P_2\precsim P_3$. Let $X\subseteq P_1$ be a finitely generated submodule. Then, by definition of $\precsim$, there is a commutative diagram
	\[
	\xymatrix{X\ar@{->}^{\phi_1}[r]\ar@/_1pc/[rr]_{\id_X} &P_2\ar@{->}^{\psi_1}[r]&P_1}.
	\]
	Note that $\phi_1(X)$ is a finitely generated submodule of $P_2$. Thus, since $P_2\precsim P_3$, there is a commutative diagram
	\[
	\xymatrix{\phi_1(X)\ar@{->}^{\quad\phi_2}[r]\ar@/_1pc/[rr]_{\id_{\phi_1(X)}} &P_3\ar@{->}^{\psi_2}[r]&P_2}
	\]
	Combining both diagrams, we define $\phi_3=\phi_2\circ \phi_1$ and $\psi_3= \psi_1\circ\psi_2$ to obtain:
	\[
	\xymatrix{X\ar@{->}^{\phi_1}[r]\ar@/_1pc/[rr]_{\phi_3} &\phi_1(X)\ar@{->}^{\phi_2}[r]&P_3\ar@{->}^{\psi_2}[r]\ar@/_1pc/[rr]_{\psi_3} &P_2\ar@{->}^{\psi_1}[r]&P_1},
	\]
	which satisfies $\psi_3\circ\phi_3=\psi_1\circ\psi_2\circ \phi_2\circ \phi_1=\psi_1\circ\id_{\phi_1(X)}\circ\phi_1=\psi_1\circ\phi_1=\id_X$.
	
	To show that $\precsim$ is compatible with the direct sum of modules, let $P$, $Q$, $P'$, $Q'\in\preCP(R)$ and suppose that $P\precsim P'$, $Q\precsim Q'$. Let $X\subseteq P\oplus Q$ be a finitely generated submodule. Then, there exist finitely generated submodules $X_{P}\subseteq P$ and $X_{Q}\subseteq Q$ such that $X\subseteq X_{P}\oplus X_{Q}$. By assumption, we have module maps $\phi_1\colon X_P\to P'$, $\psi_1\colon P'\to P$, $\phi_2\colon X_Q\to Q'$, $\psi_2\colon Q'\to Q$, and commutative diagrams
	\[
	\xymatrix{X_{P}\ar@{->}^{\phi_1}[r]\ar@/_1pc/[rr]_{\id_{X_{P}}}&P'\ar@{->}^{\psi_1}[r]&P & &
		X_{Q}\ar@{->}^{\phi_2}[r]\ar@/_1pc/[rr]_{\id_{X_{Q}}}&Q'\ar@{->}^{\psi_2}[r]&Q },
	\]
	which yields the following commutative diagram
	\[
	\xymatrix{X\subseteq X_{P}\oplus X_{Q}\ar@{->}^{\,\,\,\quad\phi_1\oplus \phi_2}[r]\ar@/_1pc/[rr]_{\id_{X_{P}\oplus X_{Q}}}& P'\oplus Q'\ar@{->}^{\psi_1\oplus \psi_2}[r]& P\oplus Q}.\qedhere
	\]
\end{proof}

\begin{remark}
	Given a finitely generated projective module $F$ and a countably generated projective module $P$, it follows from our definition that $F\precsim P$ if and only if $F$ is isomorphic to a direct summand of $P$.
\end{remark}

We are now going to show that $\CP(R)$ is order-isomorphic to $\SR(R)$. Instrumental ingredients in our proof will be the facts, proved in \cite[Lemma 4.1]{Puninski2007}, that every countably generated projective $R$-module $P$ over a unital ring $R$ is isomorphic to a direct limit of the form
\[\xymatrix{R^{n_0}\ar@{->}[r]^{x_0\cdot}&R^{n_1}\ar@{->}[r]^{x_1\cdot}&\ldots\ar@{->}[r] & P,}\]
for some sequence of positive integers ${n_i}$, where for each $x_i\in M_{n_{i+1}\times n_i}(R)$ there exists $y_{i+1}\in M_{n_{i+1}\times n_{i+2}}(R)$ such that $y_{i+1}x_{i+1}x_i=x_i$, and that, conversely, any such direct limit is always projective. We will give below independent proofs of these facts, since our arguments offer additional information that will be useful later on.

For any ring $R$, unital or not, we will denote by $\mathrm{FCM} (R)$ the ring of those $\N\times \N$ matrices $A$ with coefficients in $R$ such that each column of $A$ has only a finite number of nonzero entries. We refer to $\mathrm{FCM}(R)$ as the ring of \emph{finite-column matrices} over $R$. When $R$ is unital this ring can be identified with the ring $\mathrm{End}_R(R^{(\N)})$ of $R$-module endomorphisms of the free $R$-module $R^{(\N)}$.

\begin{lemma}
	\label{lem:projective-as-limit-of-free}
	Let $R$ be a unital ring, and let $P\in \preCP (R)$. Then $P$ is isomorphic to a projective module of the form $\varinjlim _i (R^{k_i}, Z_i\cdot )$, where $Z_i\in M_{k_{i+1}\times k_i}(R^+)$ for an increasing sequence of positive integers $(k_i)$, and moreover $Z_{i+1}Z_i =  \begin{pmatrix} Z_i \\ 0_{(k_{i+2}-k_{i+1})\times k_i}\end{pmatrix}$ for all $i\in \N$. In particular, $P$ can be written in the form $P\cong \varinjlim_i (R^{n_i},x_i\cdot)$ for some sequence of positive integers ${n_i}$, where for each $x_i\in M_{n_{i+1}\times n_i}(R)$ there exists $y_{i+1}\in M_{n_{i+1}\times n_{i+2}}(R)$ such that $y_{i+1}x_{i+1}x_i=x_i$.
\end{lemma}

\begin{proof}
	We can assume that $P =  E(R^{(\N)})$ for an idempotent $E=(e_{ij})\in \mathrm{FCM}(R)$. For each $n\ge 1$, we identify $R^n$ with $R^n \times \{0\}\times \{0\}\times \cdots $ in $R^{(\N)}$. Take an arbitrary positive integer $k_0$. Let $k_1> k_0$ be an integer such that $e_{k,l}=0$ for all $(k,l)$ such that $k>k_1$ and $l\le k_0$. In particular, one has $E(R^{k_0})\subseteq R^{k_1}$. Proceeding inductively we may find an increasing sequence of positive integers $(k_i)$ such that $e_{k,l}= 0$ for all $(k,l)$ such that $k>k_{i+1}$ and $l<k_i$. Then we have $E(R^{k_i})\subseteq R^{k_{i+1}}$ for all $i\in \N$. Let $Z_i$ be the $k_{i+1}\times k_i$ upper left corner of $E$. Thus, we get $Z_i\in M_{k_{i+1}\times k_i} (R)$ and $Z_{i+1}Z_i = \begin{pmatrix} Z_i \\ 0_{(k_{i+2}-k_{i+1})\times k_i}\end{pmatrix}$ for all $i\in \N$. Moreover $P = \bigcup_{i=0}^{\infty} Z_i R^{k_i}$ and $P\cong \varinjlim (R^{k_i}, Z_i\cdot)$, as claimed.
\end{proof}

\begin{parag}[Splittings]
	\label{par:limit-is-projective}
	Let $(n_i)$ be a sequence of positive integers, and let $(x_i)$ be a sequence, with $x_i\in M_{n_{i+1}\times n_i}(R)$, such that there exists $y_{i+1}\in M_{n_{i+1}\times n_{i+2}}(R)$ satisfying $y_{i+1}x_{i+1}x_i=x_i$. We show that $P:= \varinjlim_i (R^{n_i},x_i\cdot)$ is a countably generated projective $R$-module by exhibiting a concrete splitting of $P$ into $R^{(\N)}$. 
	
	Let $\phi_{i}\colon R^{n_{i}}\rightarrow P$ be the canonical morphisms into the direct limit, and denote by $P_{i}\subseteq P$ the image of $\phi_{i}$. Note that $P_{i}$ is a finitely generated submodule of $P$, with generators $z_{i}^{j}:=\phi_{i}(e_{j})$.
	Setting ${\bf z}_i=(z_i^1, \ldots, z_i^{n_i})$, which is a row matrix with coefficients in $P$, it follows that $P$ is described by the generators $(z^j_i)_{i,j}$, subject to the relations ${\bf z}_{i+1}x_i={\bf z}_i$ for all $i\in \N$.
	
	Further, note that for every pair $j\geq i+1$ the following diagram is commutative
	\begin{equation*}
		\xymatrixcolsep{10pc}
		\xymatrixrowsep{2pc}
		\xymatrix{
			R^{n_{j}} \ar[rd]^-{y_{i+1}\ldots y_{j}x_{j}\cdot} \ar[d]_-{x_{j}\cdot } &\\
			R^{n_{j+1}} \ar[r]_-{y_{i+1}\ldots y_{j+1}x_{j+1}\cdot } & R^{n_{i+1}}
		}
	\end{equation*}
	
	In particular, for every fixed $i$ the previous diagrams induce a morphism $g_{i}\colon P\rightarrow R^{n_{i+1}}$. Restricting to each component, we define $g_{i}^{j}:=\pi_{j}\circ g_{i}\colon P\rightarrow R$, for $j\in\{1,\ldots, n_{i+1}\}$, where $\pi_j\colon R^{n_{i+1}}\to R$ is the projection onto the $j$-th component.
	
\end{parag}


\begin{lemma}
	\label{lem:splitting} Following the above notation, a concrete splitting of $P$ into $R^{(\N)}$ is given by the formulas
	$$\pi \colon R^{(\N)}= R^{n_1}\oplus R^{n_2}\oplus \cdots \to P,\quad \iota \colon P  \to R^{(\N)}= R^{n_1}\oplus R^{n_2}\oplus \cdots ,$$
	where 	$$\pi (a_1,a_2,\dots ) = \sum_{i=1}^{\infty} \phi _i (a_i),\quad \iota (x) = (g_0(x), g_1(x)-x_1g_0(x), g_2(x)-x_2g_1(x),\dots )$$
	for $a_i\in R^{n_i}$ and $x\in P$. Then one has $\pi \circ \iota= \id_P$. In particular, $P$ is a countably generated projective right $R$-module.   
\end{lemma}

\begin{proof}
	Suppose that $x\in P_i$, and write $x= \phi_i (a)$ for $a\in R^{n_i}$. Then we have
	\begin{align*}
		\phi_{i+1} (g_i(x)) & = \phi _{i+1}(g_i (\phi_i(a))) = \phi_{i+1}(g_i(\phi_{i+1}(x_ia))) \\  
		&  = \phi_{i+1}(y_{i+1}x_{i+1}x_ia) = \phi_{i+1}(x_ia)  \\ &  = \phi_i(a) = x.
	\end{align*}
	Hence $(\phi_{i+1}\circ g_i)|_{P_i} = \id_{P_i}$, or equivalently $\phi_{i+1}\circ g_i \circ \phi_i =  \phi_i$. Using this, the identity $\pi \circ \iota = \id_P$ is easily checked. 
\end{proof}

\begin{theorem}\label{thm_CuR_iso_WR}
	Let $R$ be a unital ring. Then
	\[
	\CP(R)\cong \SR (R).
	\]
\end{theorem}
\begin{proof}
	Given a countably generated projective module $P$, it follows from  \autoref{lem:projective-as-limit-of-free} (see also \cite[Lemma 4.1(2)]{Puninski2007}) that $P\cong \lim (R^{n_i},x_i\cdot )$, where $x_i\in M_{n_{i+1}\times n_i}(R)$ are such that, for each $i$ there exists $y_{i+1}\in M_{n_{i+1}\times n_{i+2}}(R)$ such that $x_i = y_{i+1}x_{i+1}x_i$. In particular, the sequence $(x_i)$ determines an element in $\preS (R)$, through the usual identification of $x_i$ with the matrix $\text{diag}(x_{i},0,0,\ldots )$ in $M_{\infty}(R)$.
	
	We will show that the map $[P]\mapsto [(x_i)]$ defines an isomorphism between $\CP (R)$ and $\SR (R)$.
	
	First, note that this map is surjective by \autoref{lem:splitting} (see also \cite[Lemma 4.1(1)]{Puninski2007}). Moreover, it is also additive by how addition is defined in both $\CP(R)$ and $\SR(R)$; see \ref{subsec:WR}, \ref{def:PprecsimQ}.
	
	Hence, we will conclude the proof showing that $P\precsim Q$ if and only if $(x_{i})\precsim (x'_{i})$, where $P\cong \lim (R^{n_i},x_i\cdot )$ and $Q \cong \lim (R^{n'_i},x'_i\cdot )$ are the corresponding representations as direct limits.
	
	To prove the forward implication, let $P$, $Q\in\preCP(R)$ and suppose that $P\precsim Q$. Write $P=\bigcup_{i=1}^\infty P_i$ and $Q=\bigcup^{\infty}_{i=1}Q_i$ as in \autoref{par:limit-is-projective}. Then, given $i$, there are module homomorphisms $\phi\colon P_i\to Q$ and $\psi\colon Q\to P$ such that the diagram
	\[
	\xymatrix{P_i\ar@{->}^{\phi}[r]\ar@/_1pc/[rr]_{\id_{P_i}} &Q\ar@{->}^{\psi}[r]&P}
	\]
	is commutative.
	
	Since $P_{i}$ is finitely generated, so is $\phi (P_{i})$. In particular, $\phi (P_{i})\subseteq Q_{j}$ for some $j$.  By the same reasoning, one has that $\psi (Q_{j+2})\subseteq P_{l}$ for some $l$.
	
	Using the definition of the maps $g_{i},\phi_{i}$ from \autoref{par:limit-is-projective}, one gets the following commutative diagram.
	
	\[\xymatrix{
		R^{n_i} \ar@{->>}[r]^{\phi_i}\ar@{->}[d]_{x_i\cdot}                     & P_i \ar@{_{(}->}[d] \ar@{->}[r]^{\phi} & Q_j \ar@{^{(}->}[d] & R^{n_j'}\ar@{->>}[l]_{\phi_j'}\ar@{->}[d]^{x_j'\cdot} \\
		R^{n_{i+1}} \ar@{->>}[r]^{\phi_{i+1}} \ar@{->}[d]_{x_{i+1}\cdot} & P_{i+1} \ar@{_{(}->}[dd] & Q_{j+1}\ar@{^{(}->}[dd] & R^{n_{j+1}'}\ar@{->}[dd]^{x_{j+1}'\cdot}\ar@{->>}[l]_{\phi_{j+1}'}\\
		\ldots \ar@{->}[d] & & &  \\
		R^{n_l} \ar@{->>}[r]^{\phi_l} & P_l \ar@{<-}[r]^{\psi} & Q_{j+2} & R^{n_{j+2}'} \ar@{->>}[l]_{\phi_{j+2}'} }  \]
	
	Now note that, given $r\in R^{n_{i}}$ and $r'\in R^{n'_{j}}$ such that $q:=\phi\phi_{i}(r)=\phi'_{j}(r')$, we have
	\begin{equation*}
		\begin{split}
			q &= \phi'_{j}(r')= \phi'_{j+1}(x'_{j}r')=\phi'_{j+1}(y'_{j+1}x'_{j+1}x'_{j}r')\\
			&=\phi'_{j+1}g'_{j}\phi'_{j+1}(x'_{j}r')=\phi'_{j+1}g'_{j}(q) =\phi'_{j+2}x'_{j+1}g'_{j}(q) ,
		\end{split}
	\end{equation*}
	where in the fourth step we have used that $g'_{j}\phi'_{j+1}=y'_{j+1}x'_{j+1}$.
	
	Thus, one gets that
	\begin{equation*}
		\begin{split}
			x_{l-1}\ldots x_{i}(r)&=y_{l}x_{l}x_{l-1}\ldots x_{i}(r)=g_{l-1}\phi_{l}(x_{l-1}\ldots x_{i}(r))\\
			&=g_{l-1}\phi_{l-1}(x_{l-2}\ldots x_{i}(r))=\ldots =g_{l-1}\phi_{i}(r)\\
			&=g_{l-1}\psi\phi\phi_{i}(r)=g_{l-1}\psi(q)=(g_{l-1}\psi\phi'_{j+2})x'_{j+1}(g'_{j}\phi\phi_{i})(r) .
		\end{split}
	\end{equation*}
	Since this holds for every $r\in R^{n_{i}}$, we can multiply by $y_{i+1}\ldots y_{l-1}$ to obtain
	\begin{equation*}
		x_{i}=(y_{i+1}\ldots y_{l-1}g_{l-1}\psi\phi'_{j+2})x'_{j+1}(g'_{j}\phi\phi_{i}).
	\end{equation*}
	It follows that $(x_{i})\precsim (x'_{i})$, as required.
	
	For the converse, let $(x_{i})\precsim (x'_{i})$. For any fixed $i$ we need to construct a commutative diagram
	\[
	\xymatrix{P_i\ar@{->}^{\phi}[r]\ar@/_1pc/[rr]_{\id_{P_i}} &Q\ar@{->}^{\psi}[r]&P_i}.
	\]
	
	We know that, for every fixed $i$, there exist $j\in\mathbb{N}$, $\alpha \in M_{n_j'\times n_{i+1}}(R)$ and $\beta \in M_{n_{i+2}\times n_{j+1}'}(R)$ such that $x_{i+1}=\beta x'_{j}\alpha$. Thus, one gets the following commutative diagram:
	
	\[\xymatrix{
		P_i \ar@{^{(}->}[d] & R^{n_i} \ar@{->>}[l]_{\phi_i}\ar@{->}[d]^{x_i\cdot} & & \\
		P_{i+1} \ar@{^{(}->}[d] & R^{n_{i+1}} \ar[r]^{\alpha \cdot } \ar@{->>}[l]_{\phi_{i+1}} \ar@{->}[d]^{x_{i+1}\cdot} & R^{n_j'}\ar@{->>}[r]^{\phi_j'}\ar@{->}[d]^{x_j'\cdot} & Q_j\ar@{^{(}->}[d] \\
		P_{i+2} & R^{n_{i+2}} \ar@{->>}[l]^{\phi_{i+2}}  & R^{n_{j+1}'} \ar[l]^{\beta \cdot} \ar@{->>}[r]^{\phi_{j+1}'}  & Q_{j+1} }  \]
	
	Define $\phi:=\phi_j'\alpha {g_{i}}_{|_{P_{i}}}$ and $\psi=\phi_{i+2}\beta g_{j}'$. We have
	\begin{equation*}
		\psi\phi=\phi_{i+2}\beta (g_{j}'\phi_j')\alpha {g_{i}}_{|_{P_{i}}}= \phi_{i+2}\beta x'_{j} \alpha {g_{i}}_{|_{P_{i}}}= \phi_{i+2} x_{i+1} {g_{i}}_{|_{P_{i}}}=\phi_{i+1}{g_{i}}_{|_{P_{i}}}=\id_{{P_{i}}},
	\end{equation*}
	as desired.
\end{proof}

\begin{parag}[The semigroup $\CP (R)$ for non-unital $R$]\label{para:CPRNonUni}
	Let $R$ be an arbitrary ring, and let $R^+= \Z \oplus R$ be the unitization of $R$. Observe that $R$ sits as a two-sided ideal of $R^+$. We will denote by Mod-$R^+$ the category of unital right $R^+$-modules, and by AMod-$R$ the category of arbitrary $R$-modules. Note that we have an isomorphism of categories AMod-$R \cong$ Mod-$R^+$ sending an arbitrary $R$-module $M$ to the unique unital $R^+$-module whose underlying additive group is $(M,+)$ and whose multiplication is given by $x(n,r)= nx+xr$ for $x\in M$, $n\in \Z$ and $r\in R$.
	
	Recall that, for a ring $R$, we denote by $\mathrm{FCM}(R)$ the ring of finite-column matrices over $R$. For an arbitrary ring $R$, we will denote by $\preCP (R)$ the class of all countably generated unital projective right $R^+$-modules $P$ such that
	$P=PR$. The class $\preCP (R)$ agrees with the previously defined class $\preCP (R)$ whenever $R$ is a unital ring.  
	Given such module $P$, there exists an idempotent $E\in \mathrm{FCM}(R^+)$ such that $E((R^+)^{(\N)})\cong P$. Since $P=PR$, it follows that $E\in \mathrm{FCM}(R)$ and $E((R^+)^{(\N)})= E(R^{(\N)})$. Conversely, given an idempotent $E\in \mathrm{FCM}(R)$, the unital $R^+$-module $P=E(R^{(\N)})= E((R^+)^{(\N)})$ is countably generated and projective, and $P=PR$. Moreover if $Q$ is also in $\preCP (R)$, then $P\cong Q$ if and only if $E$ and $F$ are Murray-von Neuman equivalent idempotents in $\mathrm{FCM} (R)$ (see \autoref{par:WR}). This extends the well-known relation  between isomorphism classes of finitely generated unital projective $R^+$-modules $P$ such that $P=PR$ and idempotent matrices in $M_{\infty}(R)$, see e.g. \cite[$\S$ 5.1]{Goodirectlim}.
	
	Observe that, with the above notation, we have that $\preCP (R)$ is a subclass of $\preCP (R^+)$. Moreover $\preCP(R)$ is closed in $\preCP (R^+)$ under direct summands and countable direct sums. We consider the relation $\precsim$ inherited from the relation $\precsim $ which we have defined in $\preCP (R^+)$, that is, for $P,Q\in \preCP (R)$, we set $P\precsim Q$ if and only if $P\precsim Q$ in $\preCP (R^+)$. We denote by $\CP (R)$ the monoid of equivalence classes of objects in $\preCP (R)$ with respect to the relation $\precsim$. Note that $\CP (R)$ order-embeds in $\CP (R^+)$. We further define $\V^*(R)$ as the monoid of isomorphism classes of modules from $\preCP (R)$. As in the case of unital rings, we have a canonical surjective homomorphism
	$\Phi_R \colon \V^*(R)\to \CP (R)$.
	
	It is also easily checked that $\SR (R)$ order-embeds in $\SR (R^+)$. We will show that the isomorphism $\psi \colon \CP (R^+)\to \SR(R^+)$, displayed in \autoref{thm_CuR_iso_WR}, restricts to an isomorphism from $\CP (R)$ onto $\SR (R)$. 
\end{parag}

In order to obtain this result, we find a concrete realization of the idempotent matrix $E\in \mathrm{FCM}(R)$ corresponding to a sequence $(x_i)$ in $\preS (R)$.

\begin{lemma}
	\label{lem:special-rep}
	Let $R$ be a ring, and let $(x_i)\in \preS (R)$, with $x_i \in M_{n_{i+1}\times n_i}(R)$, where all $n_i$'s are positive integers. Then the countably generated projective $R^+$-module $P=\varinjlim _i ((R^+)^{n_i},x_i\cdot)$ is isomorphic to a module of the form $Q= \bigcup_i z_i (R^{(\N)})$, where $z_i \in M_{\infty}(R)$ and $z_{i+1}z_i = z_i$ for all $i\in \N$. More precisely, there is a sequence of positive integers $(k_i)$, with $k_{i+1}> k_i$ for all $i\in \N$, such that each $z_i$ is represented by a matrix $Z_i\in M_{k_{i+1}\times k_i} (R)$ and
	$Z_{i+1}Z_i = \begin{pmatrix} Z_i \\ 0_{(k_{i+2}-k_{i+1})\times k_i}\end{pmatrix}$ for all $i\in \N$, so that $P\cong \varinjlim_i ((R^+)^{k_i}, Z_i\cdot )$.  
\end{lemma}

\begin{proof}
	By \autoref{lem:splitting}, one can explicitly compute the idempotent matrix $E\in \text{End} (R^{(\N)})$ such that $E(R^{(\N)})\cong P$, where $P\cong  \lim ((R^+)^{n_i},x_i\cdot)$.
	
	Using the notation of \autoref{par:limit-is-projective}, the splitting found in \autoref{lem:splitting} gives the idempotent $E= \iota \circ \pi \in \text{End} (R^{(\N)})$. For $i\ge 1$, the column 
	$E_{i-1}$ of $E$ (with respect to the decomposition $(R^+)^{(\N)}= (R^+)^{n_1}\oplus (R^+)^{n_2}\oplus \cdots $), is given by
	$$E_{i-1} = \begin{pmatrix} y_1y_2\cdots y_ix_i \\
		y_2y_3\cdots y_ix_i -x_1y_1y_2\cdots y_ix_i\\
		y_3y_4\cdots y_ix_i - x_2y_2y_3\cdots y_ix_i\\
		\vdots \\
		y_ix_i - x_{i-1}y_{i-1}y_ix_i \\
		x_i-x_iy_ix_i\\
		{\mathbf 0}
	\end{pmatrix}.$$
	Note that the $i$-th column $E_i$ of $E$ has (at most) $i+1$ nonzero coefficients. Let $Z_{i} \in M_{(n_1+\cdots +n_{i+2})\times (n_1+\cdots + n_{i+1})}(R)$ be the matrix consisting of the upper left $(n_1+\cdots +n_{i+2})\times (n_1+\cdots + n_{i+1})$ corner of $E$. One has
	$$Z_{i+1}Z_i = \begin{pmatrix}
		Z_i \\ 0_{n_{i+3}\times (n_1+\cdots +n_{i+1})}
	\end{pmatrix}.$$
	Let $z_i\in M_{\infty}(R)$ be the infinite matrix represented by $Z_i$. Then $(z_i) \in \preS (R)$ and $z_{i+1}z_i = z_i$ for all $i\ge 1$. Moreover, we have
	$$P\cong E(R^{(\N)}) = \bigcup_{i=0}^{\infty} z_i ((R^+)^{(\N)}) \cong \varinjlim _i ((R^+)^{k_i}, Z_i\cdot),$$
	where $k_i = n_1+\cdots +n_{i+1}$ for all $i\ge 0$. 
\end{proof}


We can now obtain the following result, generalizing \autoref{thm_CuR_iso_WR}.

\begin{theorem}
	\label{thm:nonunitalSRCP}
	Let $R$ be a ring. Then there is an isomorphism
	$\SR (R) \cong \CP (R)$ such that the following diagram
	\[\xymatrix{
		\CP (R) \ar@{->}[r]^{\psi|_{\CP (R)}}\ar@{_{(}->}[d]  & \SR (R) \ar@{_{(}->}[d]  \\
		\CP (R^+)  \ar@{->}[r]^{\psi}  & \SR (R^+) } \]
	is commutative, 
	where $\psi \colon \CP (R^+)\to \SR (R^+)$ is the isomorphism defined in  \autoref{thm_CuR_iso_WR}.  
\end{theorem}

\begin{proof}
	We need to show that the restriction of the map $\psi \colon \CP (R^+)\to \SR (R^+)$ defined in the proof of \autoref{thm_CuR_iso_WR} sends $\CP (R)$ onto $\SR (R)$.

	Given $P\in \preCP (R)$, we can assume that $P =  E((R^+)^{(\N)})$ for an idempotent $E=(e_{ij})\in FCM(R)$.
	The procedure given in the proof of \autoref{lem:projective-as-limit-of-free} gives us an increasing sequence of positive integers $(k_i)$ and a sequence $(Z_i)$, with $Z_i\in M_{k_{i+1}\times k_i} (R)$ and $Z_{i+1}Z_i = \begin{pmatrix} Z_i \\ 0_{(k_{i+2}-k_{i+1})\times k_i}\end{pmatrix}$ for all $i\in \N$, such that $P \cong \varinjlim ((R^+)^{k_i}, Z_i\cdot)$. By the definition of the map $\psi$  we have that $\psi ([P]) = [(z_i)_i] \in \SR (R)$, where $z_i\in M_{\infty}(R)$ are represented by $Z_i$ for all $i\in \N$. 
	Hence $\psi ([P]) \in \SR (R)$, as desired. 
	
	Now if $w\in \SR (R)$, it follows from \autoref{lem:special-rep} that $\psi^{-1}(w)$ can be represented by 
	a countably generated unital projective right $R^+$-module of the form $P= E(R^{(\N)})$, where $E\in FCM(R)$.
	In particular, $P=PR$, so that $P\in \preCP (R)$, and hence $w= \psi (\psi ^{-1}(w))= \psi ([P]) \in \psi (\CP (R))$. This shows that $\psi (\CP (R)) = \SR (R)$, completing the proof.          
\end{proof}
The following corollary is a useful consequence of the above proof.

\begin{corollary}
	\label{cor:special-rep}
	Let $R$ be a ring. Then every element in $\SR (R)$ has a representative of the form $(z_i)$, where $z_i\in M_{\infty}(R)$ satisfy $z_{i+1}z_i = z_i$ for all $i\in \N$. More precisely there is a non-decreasing sequence of positive integers $(k_i)$ such that each $z_i$ is represented by a matrix $Z_i\in M_{k_{i+1}\times k_i} (R)$ and
	$Z_{i+1}Z_i = \begin{pmatrix} Z_i \\ 0_{(k_{i+2}-k_{i+1})\times k_i}\end{pmatrix}$ for all $i\in \N$.  
\end{corollary} 

\section{\texorpdfstring{The category $\SCu$ and the pair $\SCu (R)$}{The category SCu and the pair SCu(R)}}\label{sec:CatSCu}

As mentioned above, we will consider in this section the pair $(\RLambda(R), \SR(R))$ and show that it sits naturally in a category that we term $\SCu$ consisting of pairs $(S,W)$ where $S$ is a $\Cu$-semigroup and $W$ is a subsemigroup closed under suprema of certain sequences. We also show that the assignment $R\mapsto(\RLambda(R), \SR(R))$ is functorial.

\begin{parag}[Weakly increasing sequences]
	\label{par:weaklyinc}
	Let $S$ be a $\Cu$-semigroup. A sequence $(x_{n})$ of elements in $S$ is said to be \emph{weakly increasing} if, for every $n$ and for every $x\ll x_{n}$, there exists $m_{0}$ such that $x\ll x_{m}$ whenever $m\geq m_{0}$.
	
	It is evident that every increasing sequence in $S$ is also weakly increasing. We know that increasing sequences always have suprema in $S$, and we show below that this is also the case for weakly increasing sequences. Although the concept of a weakly increasing sequence may seem somewhat artificial, it will become key to show that the category $\SCu$ introduced in \autoref{par:SCu} admits inductive limits, as we prove in \cite{AntAraBosPerVil23arX:ContandIde}.
\end{parag}
\begin{lemma}
	\label{lma:wincsup} Let $S$ be a $\Cu$-semigroup. Then, every weakly increasing sequence has a supremum in $S$.
\end{lemma}
\begin{proof}
	We use an argument similar to the proof of \autoref{rmk:ifneedbe} (which in turn is similar to the proof that increasing sequences in a $\Cu$-semigroup have suprema). We give some details as this will be used again below.
	
	Let $(x_{n})$ be a weakly increasing sequence in $S$. Since $S$ is a $\Cu$-semigroup, we may write each $x_{m}$ as $x_{m}=\sup_{n}x_{n}^{(m)}$, where $(x_{n}^{(m)})$ is a rapidly increasing sequence. We construct increasing sequences of positive integers $(n_i)$, $(m_i)$ such that $x_{n_i+j}^{(m_i)}\ll x_{n_k}^{(m_k)}$ whenever $i+j\leq k$.
	
	To do this, we define the sequence inductively. Let $n_1=0$ and $m_1=1$. Since $x_{n_1+1}^{(m_1)}=x_1^{(1)}\ll x_1$, there is $m_2>1$ such that $x_1^{(1)}\ll x_{m_2}$, and thus there is $n_2>0$ with $x_1^{(1)}\ll x_{n_2}^{(m_2)}$. Now, assume that $n_i$, $m_i$ have been constructed for $i\leq k$. Since for each $1\leq j\leq k$ we have $x_{n_1+k-(j-1)}^{(m_j)}\ll x_{m_j}$, and the sequence is weakly increasing, there is $m_{k+1}> m_k$ such that $x_{n_1+k-(j-1)}^{(m_j)}\ll x_{m_{k+1}}$ for all $j$. Thus, there is $n_{k+1}>n_k$ such that $x_{n_1+k-(j-1)}^{(m_j)}\ll x^{(m_{k+1})}_{n_{k+1}}$ for all $j$. This completes the inductive step.
	
	After reindexing the sequence $(n_i)$, we may assume that $n_i=i$, and thus $x_{i+j}^{(m_i)}\ll x_{i+j}^{(m_{i+j})}$ whenever $i,j\geq 1$. Now, the sequence $(x_{k}^{(m_k)})$ is rapidly increasing, since $x_{k}^{(m_k)}\ll x_{k+1}^{(m_k)}\ll x_{k+1}^{(m_{k+1})}$, and one may check that its supremum is the supremum of the weakly increasing sequence.
\end{proof}
\begin{remark} Although we will not be using this, it is worth mentioning that weakly increasing sequences as defined in \autoref{par:SCu} are compatible with other properties in the category $\Cu$. Namely,
	\begin{enumerate}[(i)]
		\item $\Cu$-morphisms preserve weakly increasing sequences and their suprema.
		\item The addition in a $\Cu$-semigroup is compatible with suprema of weakly increasing sequences.
	\end{enumerate}
\end{remark}
\begin{parag}[The Category $\SCu$]
	\label{par:SCu}
	Let $S$ be a $\Cu$-semigroup. We say that a subset $H$ of $S$ is \emph{closed under suprema of weakly increasing sequences} if, given any weakly increasing sequence $(x_{n})$ in $S$ whose elements are in $H$, we have that $\sup x_{n}\in H$.
	
	We define $\SCu$ to be the abstract category whose \emph{objects} are the pairs $(S,W)$, where $S\in \Cu$ and $W$ is a submonoid of $S$ closed under suprema of weakly increasing sequences. The \emph{morphisms} in $\SCu$ are $f\colon(S_1,W_1)\to (S_2,W_2)$ where $f\colon S_1\to S_2$ is a $\Cu$-morphism such that $f(W_1)\subseteq W_2$.
\end{parag}

\begin{examples} The following are natural examples in the category $\SCu$.
	\begin{enumerate}[(i)]
		\item Any pair $(S,W)$ with $S\in \Cu$ and $W$ a sub-$\Cu$-semigroup of $S$ is an object in $\SCu$.
		\item A nonzero $\Cu$-semigroup $S$ is said to be {\it simple} if the only ideals of $S$ are $\{ 0\}$ and $S$ (see e.g. \cite{APT-Memoirs2018} for the definition of ideal in a $\Cu$-semigroup). Let $S$ be a simple $\Cu$-semigroup. Then $(S,\{ 0,\infty \})$ is an object in $\SCu$. Note that $\{ 0,\infty \}$ is not always a sub-$\Cu$-semigroup of $S$.
		\item The pair $([0,\infty],\ol{\N} )$, where $\ol{\N}= \N \cup \{\infty\}$, is an object in $\SCu$. This follows since every weakly increasing sequence with elements in $\ol{\N}\subseteq [0,\infty ]$ has an increasing cofinal subsequence. (Here, $x\ll y$ if and only if $x< y$, for $x\in [0,\infty]$ and $y\in (0,\infty]$.) However, as we have observed above, $\ol{\N}$ is not a sub-$\Cu$-semigroup of $[0,\infty]$.
	\end{enumerate}
\end{examples}


	\begin{proposition}\label{lma:SCuRinSCu} Let $R$ be a weakly $s$-unital ring. Then:
		\begin{enumerate}[{\rm (i)}]
			\item The pair $\SCu(R):=(\RLambda (R) ,\SR (R))$ is an object in $\SCu$.
			\item If $R'$ is another weakly $s$-unital ring and $f\colon R\to R'$ is a ring homomorphism, then $f$ induces a morphism $\SCu(f)\colon (\RLambda (R), \SR (R))\to(\RLambda (R'), \SR (R'))$.
		\end{enumerate}
	\end{proposition}
	\begin{proof}
		(i): By \autoref{rmk:ifneedbe}, $\RLambda(R)$ is a $\Cu$-semigroup, and by construction $\SR(R)$ is a subsemigroup of $\RLambda(R)$. We thus have to prove that $\SR (R)$ is closed under suprema of weakly increasing sequences. 
		
		Let $([x_m])$ be a weakly increasing sequence in $\RLambda (R)$ with $[x_m]\in \SR(R)$ for all $m$. In order to construct the supremum of $([x_m])$, we follow the argument in \autoref{lma:wincsup}. Write $x_m=(x_n^{(m)})$, and find $y^{(m)}_n$ such that $y^{(m)}_{n+1}x_{n+1}^{(m)}x_n^{(m)}=x_n^{(m)}$. Since $R$ is weakly $s$-unital, we have that for each $m$ the sequence $z_{m,n}=(x_1^{(m)},x_2^{(m)},\dots,x_n^{(m)},x_n^{(m)},\dots)$ satisfies that $([z_{m,n}])$ is rapidly increasing with supremum $[x_m]$ in $\RLambda(R)$ (see the proof of \autoref{rmk:ifneedbe}).
		
		Arguing as in the proof of \autoref{lma:wincsup}, we find an increasing sequence $(m_k)$ such that $x_{k+1}^{(m_k)}\precsim_1 x_{k+1}^{(m_{k+1})}$ and $\sup [x_m]=[(x_{k}^{(m_k)})]$. Now, as in \autoref{W(R)_closed}, since $x_k^{(m_k)}\precsim_1 x_{k+1}^{(m_{k+1})}$ there are elements $a_{k+1}$, $b_{k+1}$ such that
		$x_{k+1}^{(m_k)}=a_{k+1}x_{k+1}^{(m_{k+1})}b_{k+1}$. Thus
		\[
		x_k^{(m_k)}=y_{k+1}^{(m_k)}x_{k+1}^{(m_k)}x_k^{(m_k)}=y_{k+1}^{(m_k)}a_{k+1}x_{k+1}^{(m_{k+1})}b_{k+1}x_k^{(m_k)}.
		\]
		Therefore, $[(x_k^{m_k})]=[(x_k^{(m_k)}b_k)]$, and the latter belongs to $\SR(R)$, as
		\[
		x_k^{(m_k)}b_k\precsim_1 x_{k+1}^{(m_{k+1})}b_{k+1} \text{ and } x_k^{(m_k)}b_k=(y_{k+1}^{(m_k)}a_{k+1})(x_{k+1}^{(m_{k+1})}b_{k+1})(x_k^{(m_k)}b_k).
		\]
		(ii): Let $R'$ be another weakly $s$-unital ring and let $f\colon R\to R'$ be a ring homomorphism, which we can extend to a homomorphism $M_\infty(R)\to M_\infty(R')$ compatible with $\precsim_1$ and $\oplus$, also denoted by $f$.
		
		Thus, we obtain a morphism of positively ordered monoids $\W (R)\to \W (R')$ defined by $\W(f)([x])=[f(x)]$. By the arguments in \cite[Paragraph 5.5.3 and Remark 5.5.6]{APT-Memoirs2018}, the assigment $\Cu(f)\colon\RLambda(R)\to\RLambda(R')$ defined by $\Cu(f)([(x_n)])=[(f(x_n))]$ is a $\Cu$-morphism.
		
		By definition $\SR (R)$ is the submonoid of $\RLambda (R)$ of elements $[(x_1,x_2,\ldots)]$ such that for each $n$, there is $y_{n+1}$ satisfying $y_{n+1}x_{n+1}x_n=x_n$. Thus, one has
		\[
		f(y_{n+1})f(x_{n+1})f(x_n)=f(x_n)
		\]
		and, consequently, $[(f(x_1),f(x_2),\ldots)]\in \SR (R')$. Hence $\Cu(f)(\SR(R))\subseteq\SR(R')$, as desired.
	\end{proof}
	As an immediate consequence, we obtain:
	\begin{corollary}\label{cor:Scufunctor}
		Let $\mathrm{Rings}^{ws}$ be the category of weakly $s$-unital rings and ring homomorphisms. The assignment
		\[
		\begin{array}{cccc}
			\SCu\colon &{\rm Rings}^{ws}&\longrightarrow&\SCu\\
			&R&\mapsto&(\RLambda (R), \SR (R))
		\end{array}
		\]
		is a functor.
	\end{corollary}
		\begin{remark}
			In \cite{AntAraBosPerVil23arX:ContandIde} we will see that $\SCu$ is not always sequentially continuous. However, $\RLambda (R)$ always is, and $\SCu$ ends up being continuous in some relevant situations.
		\end{remark}

\section{\texorpdfstring{Compact elements in $\SR (R)$}{Compact elements and algebraicity in S(R)}}\label{sec:CompEl}
	
	We have shown in \autoref{sec:CatSCu} that for any ring $R$ we have that $\SR(R)$ is a subsemigroup of $\RLambda(R)$. Both semigroups are closed under increasing sequences, and $\RLambda(R)$ is a $\Cu$-semigroup in case $R$ is a weakly $s$-unital ring. However, the question of whether $\SR(R)$ is a $\Cu$-semigroup remains elusive, even in the weakly $s$-unital case.
	
	In this section we continue our study immersing on the \emph{way-below} relation inherited at $\SR(R)$ from $\RLambda(R)$. This helps on characterizing our construction in both the unit-regular and semilocal rings setting.
	
	\begin{parag}[Algebraic $\Cu$-semigroups and compact elements]
		\label{par:algebraic}
		Recall from \autoref{para:AbsCuSgp} that an element $x$ in a $\Cu$-semigroup $S$ is termed \emph{compact} if $x\ll x$. We also say that such a semigroup $S$ is \emph{algebraic} if every element is the supremum of an increasing sequence of compact elements.
		
		For $R$ be a weakly $s$-unital ring, if two elements $a,b\in \SR (R)$ satisfy $a\ll b$ in $\RLambda (R)$, then $a\ll b$ in $\SR (R)$. However, it is unclear when the way-below relation of $\SR(R)$ agrees with the one in $\RLambda(R)$. For example, it is conceivable that the object $(\overline{\mathbb{N}},\{ 0,\infty\})$ in $\SCu$ can be realized as $\SCu (R)$ for a weakly $s$-unital ring $R$, where $\infty\ll\infty$ in $\{ 0,\infty\}$ but $\infty\not\ll \infty$ in $\overline{\mathbb{N}}$.
		
		Keeping this type of examples in mind, for a given weakly $s$-unital ring $R$, an element $x\in\SR (R)$ is termed compact if $x\ll x$ in $\RLambda(R)$. Further, we will say that $\SR (R)$ is \emph{algebraic} if every element in $\SR (R)$ can be expressed as the supremum of an increasing sequence of compact elements. 
	\end{parag}
	\begin{lemma} Let $R$ be a weakly s-unital ring. If $\SR(R)$ is algebraic then it is a $\Cu$-semigroup. Moreover if $x\ll x$ in $\SR(R)$, then $x\ll x$ in $\RLambda(R)$. Therefore, the inclusion $\SR(R)\to \RLambda(R)$ is a $\Cu$-morphism. 
	\end{lemma}
	\begin{proof} The first assertion is clear. Given $x\ll x\in \SR(R)$, write $x=\sup_n x_n$ with $x_n\ll x_n \in \RLambda(R)$. This implies that there exists $m$ such that $x\leq x_m\ll x_m \leq x$ and hence $x\ll x$ in $\RLambda(R)$. 	
	\end{proof}
	
	
	This raises the interesting question of characterizing the elements in $\SR(R)$ that are compact. To this end, recall that, for elements $[(x_n)]$, $[(y_n)]\in\RLambda(R)$, we have $[(x_n)]\ll[(y_n)]$ in $\RLambda(R)$ if, and only if, there is $n_0$ such that $x_n\precsim_1 y_{n_0}$ for all $n$.
	
	A natural source of compact elements of $\SR(R)$ comes from the idempotent elements of $M_\infty(R)$. Indeed, if $e\in M_\infty(R)$ is idempotent, let us denote by $(e)$ the constant sequence (in $T(R)$). We clearly have that $[(e)]\in \SR(R)$ and, for another idempotent $f\in M_\infty(R)$, it is readily verified that $[(e)]\leq [(f)]$ if, and only if, $e\mvns f$.
		
	Although not all compact elements in $\SR(R)$ come from cons\-tant sequences of idempotents, we show below that there is always a representative given by a constant sequence of an \emph{almost} idempotent element.
	
	\begin{lemma}\label{lma:Alg}
		Let $R$ be a weakly $s$-unital ring and let $[(x_n)]\in \SR (R)$ be a compact element. Then,  there exists $n_0\geq 1$ such that, for every $k\geq 1$, one can find elements $s_{1},\ldots ,s_{k}$ in $M_{\infty}(R)$ satisfying
		\[
		x_{n_0} = s_k x_{n_0}\ldots s_1 x_{n_0}.
		\]
		In particular, an element $[(x_n)]$ is compact if and only if there exist elements $s,z\in M_\infty(R)$ such that $[(x_n)_n]=[(z)_n]$ and $z=sz^2$.
	\end{lemma}
	
	\begin{proof}
		Since $[(x_n)]\in\SR(R)$, there are elements $y_n$ such that $x_n=y_{n+1}x_{n+1}x_n$ for all $n$. If $[(x_n)]\ll [(x_n)]$, this implies that there exists $n_0\geq 1$ such that
		\[
		x_{n_0+k}\precsim_1 x_{n_0}\text{ for every }k \ge 1.
		\]
		
		For any given $k$, let $r_k,t_k$ be such that $x_{n_0+k}=r_k x_{n_0} t_k$. Then, using that
		\[
		x_{n_0} = (y_{n_0+1}\ldots y_{n_0+k})x_{n_0+k}\ldots x_{n_0+1}x_{n_0} ,
		\]
		we get
		\[
		x_{n_0} = ((y_{n_0+1}\ldots y_{n_0+k})r_k)x_{n_0}
		(t_k r_{k-1})x_{n_0} (t_{k-1}r_{k-2})\ldots 
		(t_1) x_{n_0}.
		\]
		Thus, if we let $s_k=(y_{n_0+1}\ldots y_{n_0+k})r_k$, $s_{k-1}=t_k r_{k-1}$, $s_{k-2}=t_{k-1}r_{k-2},\dots, s_1=t_1$, we obtain $x_{n_0} = s_k x_{n_0}\ldots s_1 x_{n_0}$, as desired.
		
		In particular, if $k=2$, set  $z=x_{n_0} s_1$ and $s=s_2$. Now
		\[
		z=x_{n_0} s_1=(s_2x_{n_0} s_1x_{n_0}) s_1=s_2 z^2=sz^2,
		\]
		and clearly the constant sequence $[(z)_n]$ belongs to $\SR(R)$. Note that  $x_{n_0}=s_2(x_{n_0}s_1)x_{n_0}$ implies $x_{n_0}\precsim_{1} x_{n_0}s_1=z$, and also that $z=sz^2=s(x_{n_0}s_1)z$. Hence $z\precsim_1 x_{n_0}$. Therefore,
		\[ x_{n_0+k}\precsim_1 x_{n_0}\precsim_1 x_{n_0} s_1=z \precsim_1 x_{n_0},\]
		which implies that $[(x_n)_n]=[(z)_n]$.
		
		Finally, if $z\in M_\infty(R)$ satisfies $z=sz^2$ for some $s$, then clearly $z\precsim_1 z$ and therefore $[(z)]\ll [(z)]$.
	\end{proof}
	\begin{remark}\label{rmk:auxiliary_Lamdaseq}
		It is reasonable to extend results such as \autoref{lma:Alg} to general rings. To do this, one may define a transitive relation $\prec$ on $\RLambda(R)$ (or on $\SR(R)$) as follows:
		$$[(x_n)]\prec [(y_n)]\text{ if and only if there is } m\text{ such that } x_n\precsim_1 y_m\text{ for all }n.$$
		
		Using the construction of suprema in $\RLambda(R)$ (see the proof of \autoref{rmk:ifneedbe}), it is easy to verify that $\prec$ is formally stronger than the way-below relation $\ll$ on $\RLambda(R)$, and of course it agrees with it in case $R$ is weakly $s$-unital. Thus, one may term an element $x\in \SR(R)$ \emph{$\prec$-compact} in case $x\prec x$.
		
		It is also worth pointing out that the relation $\prec$ is an \emph{auxiliary relation} for the usual order in $\SR(R)$ (and also $\RLambda(R)$). Following \cite[Definition I-1.11]{GieHof+03Domains} (see also \cite[2.1.1]{APT-Memoirs2018}), an auxiliary relation is a relation satisfying that $0\prec x$ for any $x$, that $x\leq y$ whenever $x\prec y$, and whenever $x\leq y\prec z\leq u$, we have $x\prec u$.
		
		A close inspection of the arguments in \autoref{lma:Alg} reveals that, for any ring $R$, an element $[(x_n)]$ in $\SR(R)$ is $\prec$-compact if and only if $[(x_n)]=[(z)]$, where $z=sz^2$ for some $s$.
	\end{remark}
	As the example below shows, certain rings have very few $\prec$-compact elements.
	\begin{example}
		There exists a ring $R$ such that $x\prec x$ in $\SR(R)$ if and only if $x=0$.
	\end{example}
	\begin{proof} 
		Let $R$ be the ring in \autoref{rmk:LeftWR}. That is, $R$ is the subring of the free algebra $K[x_0, x_1,\ldots ]$ on infinitely many variables subject to the non-commutative relations $x_{n+1}x_n =x_n$ consisting of all polynomials with zero constant term.
		
		We claim that the only compact element of $\SR (R)$ is $0$.
		To show this, we need some easily proven facts about $R$. First of all, observe that the set
		$$\mathcal B = \{ x_{i_1}^{n_1}\cdots x_{i_r}^{n_r}  \mid n_i \ge 1,\, i_1<\cdots < i_r,\, r\ge 1  \} $$
		is a $K$-basis of $R$. This follows for instance by an immediate application of the Diamond's Lemma in Ring Theory, see \cite{Bergman-Diamond}, using the reduction system $x_jx_i\mapsto x_i$ for $j>i$. 
		
		Hence each element in $M_{\infty}(R)$ can be uniquely written as a linear combination $\sum_{p\in \mathcal B} a_pp$, where $a_p\in M_{\infty}(K)$. Recall from \autoref{rmk:LeftWR} the notion of the start $\text{st}(p)$ of a monomial $p\in \mathcal B$. The numbers $\text{st}(p)$ satisfy the following properties, some of which have been pointed out in \autoref{rmk:LeftWR}:
		\begin{enumerate}
			\item[(i)] $\text{st}(pq) = \text{min} \{ \text{st}(p) , \text{st}(q)\}$, 
			\item[(ii)] If $\text{st}(p) = \text{st}(q)$ then $ pq >p $ and $pq>q$ in the lexicographic order. 
			\item[(iii)] If $\text{st}(p) >  \text{st}(q) $ then $pq= q$.  
		\end{enumerate}
		
		Using these properties we now show that the only compact element of $\SR (R)$ is $0$.
		By \autoref{lma:Alg}, it is enough to show that the equation $z= sz^2$ has no nonzero solutions in $M_{\infty}(R)$. Suppose that $z\in M_{\infty}(R)$ is a nonzero element such that $z= sz^2$. Let $p$ be the unique monomial in $\mathcal B$ in the support of $z$ such that $\text{st}(p)$ is maximum amongst all monomials in the support of $z$, and such that $p$ is the smallest monomial in the support of $z$ amongst all monomials $q$ in the support of $z$ such that $\text{st} (q) = \text{st}(p)$, with respect to the lexicographic order. From the identity $z= sz^2$, it follows that there are two monomials $p_1,p_2$ in the support of $z$, and a monomial $q$ in the support of $s$ such that 
		$$p= qp_1p_2.$$
		By (i) we have
		$$\text{st} (p) = \text{min} \{ \text{st}(q),\text{st} (p_1), \text{st} (p_2)\}.$$
		It follows that  $\text{st}(p)= \text{st}(p_1) = \text{st}(p_2) \le \text{st}(q)$. 
		Now by (ii) we have $p_1p_2 > p_i \ge p$ in the lexicographic order, for $i=1,2$, and by (ii),(iii) we have, since $\text{st}(q)\ge \text{st}(p_1p_2)$, 
		$$ p = q(p_1p_2) \ge p_1p_2 > p,$$
		which is a contradiction. 
		
		We remark that the ring $R$ is not weakly s-unital. Indeed, if $x_0=rx_0s$ for some $r,s\in R$, one easily gets a contradiction expressing $r$ and $s$ in terms of the $K$-basis $\mathcal B$.
	\end{proof}
	
	
	\begin{proposition}\label{prp:CuRAlg}
		Let $R$ be a unital ring. Then, $\SR  (R)$ is an algebraic $\Cu$-semigroup whenever every projective module of $R$ is the direct sum of finitely generated modules.
		
		In particular, this is the case for weakly semi-hereditary rings, one-sided principal ideal rings and $R=C(X)$ for any strongly zero dimensional $X$. 
	\end{proposition}
	\begin{proof}
		Following the observations in \autoref{par:algebraic}, we only need to show that every element in $\SR(R)$ is the supremum of an increasing sequence of compact elements. For this, we will use the isomorphism between $\SR (R)$ and $\CP (R)$ proved in \autoref{thm_CuR_iso_WR}.
		
		First, note that any finitely generated projective module $P$ has an associated sequence in $\preS (R)$ of the form $(e,e,e,\dots)$, where $e=e^2\in M_{\infty}(R)$. As we have observed before, the class of such a sequence is compact in $\SR(R)$, and thus the class $[P]$ is compact in $\CP(R)$.
		
		Now, let $P$ be a countably generated projective module. From our assumptions on $R$, we may write $P=\oplus F_{i}$ with $F_{i}$ finitely generated (and projective). We have $F_{1}\precsim F_{1}\oplus F_{2}\precsim\ldots$.
		
		This shows that $[P]=\sup_n [\oplus_{i\leq n}F_i]$ is the supremum of an increasing sequence of compact elements in $\CP (R)$, as desired.
		
		The remaining statement is a consequence of the results in \cite{Puninski2007}.
	\end{proof}
	
	\begin{example}\label{exa:NotComp}
		As an example of a ring that does not satisfy the condition in  \autoref{prp:CuRAlg}, let $R=C[0,1]$.
		Then $R$ has an indecomposable, countably projective and pure ideal (which is not finitely generated). See \cite[Example 2.12]{Lam1999}.
	\end{example}
	
	We now do a more in-depth study of the semigroup $\SR (R)$ when $R$ is a unit-regular ring (\autoref{lma:OrderRegRings}) and when $R$ is a semilocal ring (\autoref{prp:OrderSemi}). Since $\SR (R)\cong \CP (R)$, we use these two pictures interchangeably.
	
	\begin{parag}[Unit-regular rings]

		Recall that a unital ring $R$ is said to be {\it unit-regular} if for each $x\in R$ there is an invertible element $u\in R$ such that $x=xux$. Unit-regular ring are precisely those unital rings $R$ such that $\V(R)$ is cancellative (\cite[Theoremm 4.5]{vnrr}).  
		
		We first observe that for a unit-regular ring $R$, the relation $P\precsim Q$ in $\preCP (R)$ is determined solely in terms of isomorphisms of all finitely generated submodules of $P$ with suitable submodules of $Q$.  
		
		\begin{lemma}\label{lma:OrderRegRings}
			Let $R$ be a unit-regular ring and let $P,Q$ be countably generated projective modules. Then, $P\precsim Q$ if and only if every finitely generated submodule of $P$ is isomorphic to a submodule of $Q$.
		\end{lemma}
		
		\begin{proof}
			Given any unital ring $R$, it follows from \autoref{def:PprecsimQ} that whenever $P\precsim Q$, then every finitely generated submodule of $P$ is isomorphic to a submodule of $Q$.
			
			Conversely, assume that $R$ is a unit-regular ring, and that every finitely generated submodule of $P$ is isomorphic to a (finitely generated) submodule of $Q$. Then $P\precsim Q$ follows from   \autoref{def:PprecsimQ} and the fact that all finitely generated submodules of $Q$ are direct summands of $Q$ (\cite[Theorem 1.11]{vnrr}). 
		\end{proof}
		
		Next example exhibits that above characterization does not hold in general.
		\begin{example}
			\label{example:counterexample}
			There exist unital commutative domains $R$ and countably ge\-ne\-rated projective modules $P$ and $Q$ such that all finitely generated submodules of $P$ are isomorphic to submodules of $Q$ but $P\precsim Q$ does not hold. 
		\end{example}
		
		\begin{proof}
			Let $R$ be a commutative domain with an indecomposable, countably ge\-ne\-rated projective module $Q$, which is not free, and take $P=R_R$. Then obviously $R$ is isomorphic to a submodule of $Q$. If there is a commutative diagram
			\[
			\xymatrix{R\ar@{->}^{\phi}[r]\ar@/_1pc/[rr]_{\id_R} &Q\ar@{->}^{\psi}[r]&R},
			\]	
			then $Q\cong R\oplus Q'$ for some projective module $Q'$, which is impossible since $Q$ is indecomposable and non-free.      
		\end{proof}
		
		
		For a unit-regular ring, all semigroups already defined turn out to be isomorphic to either $\V(R)$ or $\SR(R)$.    
		
		\begin{proposition}\label{prp:CuRegular}
			Let $R$ be a unit-regular ring. Then we have 
			\begin{enumerate}[{\rm (i)}]
				\item $\V(R) = \W(R) = \W_\mathrm{M}(R)$ as ordered monoids, so that the orders defined in $\W(R)$ and $\W_\mathrm{M}(R)$ agree with the algebraic order.
				\item $\V^*(R) = \CP (R) = \SR(R) = \WLambda (R) $ as semigroups. The order $\precsim$ in $\CP (R)$ is given by: $P\precsim Q$ if and only if $P$ is isomorphic to a submodule of $Q$. We have that $\SR(R)= \CP (R)= \RLambda (R)$ as ordered semigroups.
			\end{enumerate} 
		\end{proposition}

		\begin{proof}
			(i): This follows from Lemmas 2.6 and 3.3.
			
			(ii): By \cite[Theorem 1.4]{APP2000}, there is a monoid isomorphism
			$$\Upsilon  \colon \V^*(R) \longrightarrow \Lambda_{\sigma} (\V(R)).$$
			(We warn the reader that the monoid $\V^*(R)$ is denoted by $\W(R)$ in \cite{APP2000}.) 
			By (i), one has that $\V(R)= \W(R)$, so $\Lambda_{\sigma} (\V(R)) = \Lambda_{\sigma} (\W(R)) = \WLambda (R)$. The isomorphism $\Upsilon$ satifies that
			$\Upsilon (P)$ is the interval determined by the increasding sequence $ \{ [e_1R\oplus \cdots \oplus e_nR] : n\ge 1\}$ in $\V(R)$, where $P= \bigoplus_n e_nR$ and $e_n$ are idempotents of $R$. Hence $\Upsilon$ factors as the composition of homomorphisms
			\[ 
			\V^*(R) \overset{\Phi_R}{\longrightarrow} \CP (R) = \SR(R) \overset{\iota_R}{\longrightarrow} \WLambda (R) ,
			\] 
			where $\Phi_R$ is the canonical surjective homomorphism, and $\iota_R\colon \SR(R)\to \WLambda (R)$ is the natural inclusion. It follows that both $\Phi_R$ and $\iota_R$ are monoid isomorphisms. By \cite[Proposition 1.5]{APP2000} the order in $\CP(R)$ is determined by $[P] \le [Q]$ if and only if $P$ is isomorphic to a submodule of $Q$. (Observe that this is {\it not} the algebraic order in $\CP (R)$.)
		\end{proof}
	\end{parag}
	\begin{parag}[Semilocal rings]
		Recall that a unital ring $R$ is said to be semilocal if the quotient $R/J(R)$ is semisimple, i.e. if there exist divison rings $D_{1},\ldots ,D_{r}$ such that
		\begin{equation*}
			R/J(R)\cong M_{n_{1}}(D_{1})\times\ldots\times M_{n_{r}}(D_{r}).
		\end{equation*}
		
		Observe that we have
		\[
		\V^*(R/J(R)) = \CP (R/J(R)) = \SR(R/J(R))= \WLambda (R/J(R)) = \ol{\N}^{r},
		\]
		where the order here is the algebraic order, or equivalently, the componentwise order. The generators are the isomorphism classes of the simple $R/J(R)$-modules.  
		
		Moreover $\V^*(R)$ embeds in $\V^*(R/J(R))$ by Prihoda's Theorem \cite[Theorem 2.3]{Prihoda2007}.  
		
		Note that we also have a surjective homomorphism of ordered monoids 
		\[
		\pi \colon \W(R)\to \W(R/J(R))= \V(R/J(R))= \N^r,
		\]
		which extends to a surjective homomorphism 
		\[
		\pi \colon \WLambda (R) = \Lambda_{\sigma}(\W(R)) \to \WLambda (R/J(R))= \ol{\N}^r.
		\]
		
		Now we characterize the  equivalence relation $\sim $ on $\preCP (R)$ using the notion of dimension in the case of semilocal rings (\cite{Prihoda2007}). Recall that we define $\dim (P)=(x_{1},\ldots , x_{r})\in \ol{\N}^r$, where $(x_{1},\ldots , x_{r})$ is the image of $[P]$ under the map $\V^*(R)\to \ol{\N}^r$. Further, given two countably generated projective right $R$-modules $P,Q$, we say that $\dim(P)\leq \dim(Q)$ if the corresponding tuples compare componentwise. In this case, there exists a split $R/J(R)$-monomorphism $P/PJ(R)\rightarrow Q/QJ(R)$.
	\end{parag}

	Using the remarks above, we can easily characterize the order relation in $\CP(R)$ in the case of a semilocal ring:
	
	\begin{proposition}\label{prp:OrderSemi}
		Let $R$ be a semilocal ring and let $P,Q$ be two countably ge\-ne\-rated right $R$-modules. The following are equivalent:
		\begin{enumerate}[{\rm (i)}]
			\item $P\precsim Q$
			\item $\dim (P)\leq \dim (Q)$ (component-wise)
			\item $P$ is isomorphic to a pure submodule of $Q$.
		\end{enumerate}
	\end{proposition}
	\begin{proof}
		(i)$\implies$(ii): Let $\mathbf{x} = (x_1,\dots , x_r)\in \N^r$ such that $\mathbf{x} \le \dim (P)$. Let $\iota_R \colon \CP (R)\to \WLambda (R)$ be the canonical inclusion. Then $\iota_R ([P]) \subseteq \iota_R ([Q])$ (as elements in $\Lambda_{\sigma}(\W(R))$). Take $z\in \W(R)$ such that $z \in \iota_R ([P])$ and $\pi(z)= \mathbf{x}$. Then $z\in \iota_R ([Q])$, which means that $\pi (z) \le \pi (\iota_R[Q]) = \pi ([Q]) = \dim (Q)$. It follows that $\dim (P)\le \dim (Q)$.   
		
		(ii)$\implies$(iii): Assume that $\dim (P)\leq \dim (Q)$. Then, one can construct a split monomorphism $\bar{r}\colon P/PJ(R)\rightarrow Q/QJ(R)$. Let $\bar{s} \colon Q/QJ(R)\rightarrow P/PJ(R)$ be such that $\bar{s}\bar{r}=\id_{P/PJ(R)}$.
		
		Since $P,Q$ are projective, both maps can be lifted to $r \colon P\rightarrow Q$ and $s \colon Q\rightarrow P$.
		
		Moreover, since $\bar{s}\bar{r}=\id_{P/PJ(R)}$, we know from \cite[Lemma 2.1]{Prihoda2007}, applied to $sr$, that for any finite subset $X\subseteq P$ there exists a morphism $g\colon P\rightarrow P$ such that $gsr(x)=x$ for every $x\in X$. Hence $h:= gs\colon Q\to P$ satisfies that $hr(x)= x$ for all $x\in X$.
		It follows that $r$ is injective and $r(P)$ is pure in $Q$ (see \cite[Exercise \S 4.38]{Lam1999}).

		(iii)$\implies$(i): See  \cite[Exercise \S 4.41]{Lam1999}.
	\end{proof}
	
	\begin{corollary}\label{cor:SemLocRin}
		Let $P,Q$ be two countably generated projective right modules over a unital semilocal ring $R$. Then, $P\sim Q$ if and only if $P\cong Q$. That is, $\CP  (R)\cong \V^ {*}(R)$ as semigroups.
	\end{corollary}
	
	
	\begin{corollary}
		Let $R$ be a unital semilocal ring. Then, $\CP (R)$ can be embedded into $\ol{\N}^{r}$ as a partially ordered monoid.
	\end{corollary}
	

\section{\texorpdfstring{\Cs{s}}{C*-algebras}}\label{sec_Cs}

Let $A$ be a \Cs. In this section we explore the relationship between the Cuntz semigroup $\Cu(A)$ of $A$ and the semigroup $\SR(A)$. We show in \autoref{thm:CuCs} that $\Cu(A)$ is a retract of $\SR(A)$, in the sense that there is an ordered monoid morphism $\Cu(A)\to\SR(A)$ that preserves suprema, compact containment, and has a left inverse that preserves suprema.

\begin{remark}\label{W(A)_Rmk}
	Given $f,g\in C([0,\infty))$ such that $\text{supp}(f)=(\varepsilon,\infty)$ and $\text{supp}(g)=(\varepsilon ',\infty)$ with $\varepsilon '<\varepsilon $, there exists $r\in C([0,\infty))$ such that $f(t)=r(t)g(t)f(t)$ for each $t\in [0,\infty)$. In particular, we have $f\precsim_{1} g$. In general, if $\mathrm{supp}(f)\subseteq \mathrm{supp}(g)$, then $f\precsim g$ (see, e.g. \cite[Proposition 2.5]{APT2011})
\end{remark}

\begin{parag}[Dense subsemigroups]
	\label{def:dense-subsemigroup} Let $S$ be a $\Cu$-semigroup. We will say that a subsemigroup $H$ of $S$ is a {\it dense subsemigroup} provided that whenever $x\ll y$ in $S$ there exists $s\in H$ such that $x\le s \le y$.
	
	For example, if $S$ is algebraic, then the subsemigroup $S_c$, consisting of the compact elements in $S$, is a dense subsemigroup.
\end{parag}

\begin{lemma}\label{lma:extensions}
	Let $S$ be a $\Cu$-semigroup, $T$ a positively ordered monoid where suprema of increasing sequences exist and are compatible with addition. Let $H\subseteq S$ be a dense submonoid. Then, for any ordered monoid morphism $\varphi\colon H\rightarrow T$ that preserves suprema of increasing sequences with supremum in $H$, there exists an ordered monoid morphism $\phi\colon S\rightarrow T$ that preserves suprema of increasing sequences and $\phi|_{H}=\varphi$.
\end{lemma}
\begin{proof}
	Given any $s\in S$, write $s=\sup s_{n}$ where $(s_{n})$ is a rapidly increasing sequence of elements in $S$. Since by assumption $H$ is dense in $S$, there is for each $n$ an element $s_n'\in H$ such that $s_n\le s_n'\le s_{n+1}$. Thus we may assume that $s_n\in H$ for all $n$. Define $\phi(s):=\sup_{n}\varphi (s_{n})$.
	
	If $(t_n)$ is another rapidly increasing sequence of elements in $H$ such that $\sup s_n\leq \sup t_n$, then for any $n$, there is $m$ with $s_n\leq t_m$, whence $\sup_{n}\varphi (s_{n})\leq\sup_{n}\varphi (t_{n})$. This implies that $\phi$ is well-defined and order-perserving. To see that $\phi$ preserves addition, let $s,t\in S$ and write $s=\sup_n s_n$, $t=\sup_n t_n$, for rapidly increasing sequences $(s_n)$ and $(t_n)$ in $H$. Thus, using our assumption on $T$, we get
	\[
	\phi(s+t)=\sup_n(\varphi(s_n)+\varphi(t_n))=\sup_n \varphi(s_n)+\sup_n\varphi (t_n)=\phi(s)+\phi(t), 
	\]
	and therefore $\phi$ is an ordered monoid morphism.
	
	Further, note that for every $h\in H$, we can write $h=\sup_{n}h_{n}$ for a rapidly increasing sequence $(h_{n})$ of elements in $H$. Since $\varphi$ preserves the supremum of such sequences, we have $\phi (h)=\sup_{n}\varphi (h_{n})=\varphi (h)$, and thus $\phi|_{H}=\varphi$.
	
	To see that $\phi$ preserves suprema, fix $s\in S$ and let $(t_{n})$ be an increasing sequence in $S$ with supremum $s$. Choose $(s_{n})$ to be a rapidly increasing sequence of elements in $H$ with supremum $s$. Since $\phi$ is order-preserving, we have that $\sup_n\phi(t_n)\leq \phi(s)$. Also, for every $n$, there is $m$ with $s_{n}\leq t_{m}$.  Using that $\phi|_{H}=\varphi$ and that $\phi$ is order-preserving, we obtain $\varphi(s_n)=\phi(s_n)\leq\phi(t_m)\leq\sup_k\phi(t_k)$, which implies that $\phi (s)\leq \sup_{k}\phi (t_{k})$ as required.
\end{proof}

\begin{parag}[Retracts]
	Let $S$ be a $\Cu$-semigroup, and let $T$ be a positively ordered semigroup admitting suprema of increasing sequences, which are compatible with addition. Adapting the definition introduced in \cite[Definition~3.14]{ThiVil22DimCu}, we shall say that $S$ is a \emph{retract} of $T$  if there exist  ordered monoid morphisms $\varphi\colon S\to T$ and $\phi\colon T\to S$ such that
	\begin{enumerate}[{\rm (i)}]
		\item $\varphi$ preserves suprema and compact containment.
		\item $\phi$ preserves suprema.
		\item $\phi\varphi=\id_S$.
	\end{enumerate}
\end{parag}

Given $\varepsilon>0$, we shall denote by $f_\varepsilon \in C([0,\infty))$ the continuous function that is $0$ on $[0,\varepsilon)$, linear on $[\varepsilon ,2\varepsilon]$, and $1$ elsewhere. Notice that, for each positive element $a$ in a \Cs{} $A$, the element $f_\varepsilon(a)$ has a unit in $A$, namely $f_{\varepsilon/2}(a)f_\varepsilon(a)=f_\varepsilon(a)$.
\begin{lemma}
	\label{lma:varepsilon}
	Let $A$ be a \Cs{} and let $a,b\in M_\infty(A)_+$ be such that $a\precsim b$ in $A\otimes\mathcal{K}$. Put $\varepsilon_n=1/2^n$. Then, for any $n$, there is $m$ such that $f_{\varepsilon_n}(a)\precsim_1f_{\varepsilon_{m}}(b)$.
\end{lemma}
\begin{proof}
	Since $a\precsim b$ in $A\otimes \mathcal{K}$, given $n\in\mathbb{N}$, there exist $\delta_n >0$ and $r_n\in A\otimes\mathcal{K}$ such that
	\begin{equation*}
		(a-\varepsilon_{n+1})_{+}=r_n(b-\delta_n)_{+}r_n^{*}
	\end{equation*}
	(see, for instance, \cite[Proposition 2.17]{APT2011}). As $a,(b-\delta_n)_{+}\in M_{\infty}(A)$, we can take $r_n\in M_{\infty}(A)$ as well. Since $\delta_n >0$, there exists $m$ such that $f_{\varepsilon_{m}}(b)$ is a unit for $(b-\delta_n)_{+}$. Using this observation at the third step and \autoref{W(A)_Rmk} at the first step,  one gets
	\begin{equation*}
		f_{\varepsilon_{n}}(a)\precsim_{1} (a-\varepsilon_{n+1})_{+}=r_n(b-\delta_n)_{+}f_{\varepsilon_{m}}(b)r_n^{*}\precsim_{1}f_{\varepsilon_{m}}(b).\qedhere
	\end{equation*}
\end{proof}

\begin{theorem}\label{thm:CuCs}
	Let $A$ be a \Cs{}. Then $\Cu(A)$ is a retract of $\SR(A)$. 
\end{theorem}
\begin{proof}
	Let $H=\{ x\in \Cu (A) \colon x=[a]\text{ for some }a\in M_{\infty}(A)_+ \}$, which is a dense subsemigroup of $\Cu  (A)$. We will apply  \autoref{lma:extensions} to $H$. To this end we need to construct a positively ordered monoid morphism $\varphi_0\colon H\to \SR(A)$ that preserves suprema of increasing sequences.
	
	Let $x\in H$, and let $a\in M_\infty(A)_+$ be such that $x=[a]$. For every $n\geq 1$, put $\varepsilon_n=1/2^n$ and $f_{\varepsilon_n}$ as in  \autoref{lma:varepsilon}. The elements of the sequence $(f_{\varepsilon_n}(a))$ are pairwise commuting and, as observed before, $f_{\varepsilon_{n+1}}(a)f_{\varepsilon_n}(a)=f_{\varepsilon_n}(a)$ for all $n$. Thus $f_{\varepsilon_n}(a)\precsim_1 f_{\varepsilon_{n+1}}(a)$ for all $n$, and so $(f_{\varepsilon_n}(a))_n\in\preS(M_{\infty}(A))$.

	Define $\varphi_0\colon H\rightarrow \SR (A)$ by $\varphi_0([a])=[(f_{\varepsilon_{n}}(a))_{n}]$. To see that $\varphi_0$ is well defined and order-preserving, let $b\in M_\infty(A)_+$ be such that $a\precsim b$ in $A\otimes\mathcal{K}$. By \autoref{lma:varepsilon}, for each $n$ there is $m$ with $f_{\varepsilon_n}(a)\precsim_1 f_{\varepsilon_{m}}(b)$, which implies that $[(f_{\varepsilon_n}(a))_n]\leq [(f_{\varepsilon_n}(b))_n]$.
	
	It is easy to verify that $\varphi_0$ is additive, and thus it is a positively ordered monoid morphism. In order to extend $\varphi_0$ to a positively ordered monoid morphism $\varphi\colon\Cu (A)\to \SR (A)$ that preserves suprema of increasing sequences, we apply \autoref{lma:extensions}. Thus, we only need to prove that the map $\varphi_0\colon H\to \SR (A)$ preserves suprema of increasing sequences with supremum in $H$.
	
	To this end, let $(a_{n})_{n},a\in M_{\infty}(A)_+$ be such that $([a_{n}])_n$ is increasing and $[a]=\sup_{n}[a_{n}]$ in $\Cu (A)$. Since $\varphi_0$ is order-preserving, we already have $\sup_n \varphi_0 ([a_n])\leq \varphi_0 ([a])$. To show the converse inequality, note that for every  $n\geq 1$ one has $[(a-\varepsilon_{n+1})_+]\ll [a]$ and, consequently, $(a-\varepsilon_{n+1})_+\precsim a_{i}$ for some $i$. By  \autoref{lma:varepsilon}, there is $m$ such that $f_{\varepsilon_{n+2}}((a-\varepsilon_{n+1})_+)\precsim_1 f_{\varepsilon_m}(a_i)$. Let $g_{\varepsilon_n}(t)=(t-\varepsilon_n)_+$. By \autoref{W(A)_Rmk}, we have $f_{\varepsilon_n}\precsim_1 f_{\varepsilon_{n+2}}\circ g_{\varepsilon_{n+1}}$. Thus,
	\[
	f_{\varepsilon_n}(a)\precsim_1 f_{\varepsilon_{n+2}}((a-\varepsilon_{n+1})_+)\precsim_1 f_{\varepsilon_m}(a_i).
	\]
	This shows that $\varphi_0 ([a])\leq \sup_n \varphi_0 ([a_n])$, as every element in $(f_{\varepsilon_{n}}(a))_n$ is $\precsim_1$-bounded by an element in $(f_{\varepsilon_{m}}(a_i))_m$ for some $i$. Therefore, $\varphi_0([a])=\sup_n\varphi_0([a_n])$, as was to be shown, and we have an extension $\varphi\colon \Cu(A)\to \SR(A)$ that preserves suprema of increasing sequences.
	
	Next, to prove that $\varphi$ preserves the way-below relation, take first $[a],[b]\in \Cu(A)$ with $a,b\in M_{\infty}(A)_+$ and suppose that $[a]\ll [b]$ in $\Cu (A)$. Then there is $\varepsilon>0$ such that $a\precsim (b-\varepsilon)_+$. Since again by \autoref{W(A)_Rmk}, we have $f_{\varepsilon_m}\circ g_\varepsilon\precsim_1 f_\varepsilon$ for each $m$, another usage of \autoref{lma:varepsilon} implies that, for each $n\geq 1$, there is $m\geq 1$ such that
	\[
	f_{\varepsilon_n}(a)\precsim_1 f_{\varepsilon_m}((b-\varepsilon)_+)\precsim_1 f_\varepsilon(b),
	\]
	and therefore $\varphi_0 ([a])=[(f_{\varepsilon_n}(a))_n]\ll [(f_{\varepsilon_m}(b))_m]=\varphi_0 ([b])$. If now $a,b\in (A\otimes\mathcal{K})_+$ satisfy $[a]\ll [b]$ in $\Cu(A)$, then as before there is $\varepsilon>0$ such that $a\precsim (b-\varepsilon)_+$. Note that there is $b_\varepsilon\in M_\infty(A)_+$ such that $(b-\varepsilon)_+\sim b_\varepsilon$, and since $[b_\varepsilon]\ll [b_{2\varepsilon}]$, we have 
	\[
	\varphi([a])\leq \varphi([(b-\varepsilon)_+])=\varphi_0([b_\varepsilon])\ll\varphi_0([b_{2\varepsilon}])\leq \varphi([b]).
	\]
	To finish the proof, we have to construct an ordered monoid morphism $\phi\colon \SR (A)\to \Cu (A)$ that preserves suprema of increasing sequences and is a left inverse for $\varphi$. To do this, let $(a_{n})_{n}\in\preS (M_{\infty}(A))$. By \autoref{Cucls_Rmk}, the sequence $(a^{*}_{n}a_{n})_{n}$ is $\precsim$-increasing and we can consider $[a]=\sup_{n}[a^{*}_{n}a_{n}]$ in $ \Cu (A)$. 
	
	Define $\phi\colon \SR (A)\rightarrow \Cu (A)$ by $\phi([(a_{n})_{n}])= \sup_{n}[a^{*}_{n}a_{n}]$. Let $[(a_n)], [(b_n)]\in \SR(A)$ be such that $[(a_{n})_{n}]\leq [(b_{n})_{n}]$ in $\SR (A)$. Then, for each $n$, there is $m$ such that $a_n\precsim_{1}b_{m}$. Again by \autoref{Cucls_Rmk}, this implies that $a^{*}_{n}a_{n}\precsim b^{*}_{m}b_{m}$. Therefore, if we put $[a]=\sup_{n}[a^{*}_{n}a_{n}]$ and $[b]=\sup_{n}[b^{*}_{n}b_{n}]$ in $\Cu (A)$, we obtain that $[a]\leq [b]$. This shows that $\phi$ is well defined and order-preserving. It is easy to verify that $\phi$ is also additive, hence a positively ordered monoid morphism.
	
	Now, given an increasing sequence $([(a_{i,n})_{i}])_{n}$ in $\SR (A)$, there is a subsequence $(n_i)$ of the natural numbers and elements $r_i\in A$ such that $\sup_n [(a_{i,n})_i]=[(a_{i,n_{i}}r_{i})_{i}]$ (see \autoref{W(R)_closed}). Let $[a_{n}]=\phi ((a_{i,n})_{i})$. We have, for each $i$,
	\begin{equation*}
		r^{*}_{i}a^{*}_{i,n_{i}}a_{i,n_{i}}r_{i}\precsim a^{*}_{i,n_{i}}a_{i,n_{i}}\precsim a_{n_i},
	\end{equation*}
	and thus $[r^{*}_{i}a^{*}_{i,n_{i}}a_{i,n_{i}}r_{i}]\leq [a_{n_i}]\leq \sup_n [a_n]$. Therefore
	\[
	\phi(\sup_n [(a_{i,n})_i])=\phi([(a_{i,n_{i}}r_{i})_{i}])=\sup_i[r^{*}_{i}a^{*}_{i,n_{i}}a_{i,n_{i}}r_{i}]\leq \sup_n [a_n]=\sup_n\phi([(a_{i,n})_i]).
	\]
	Since $\phi$ is order-preserving we always have $\sup_{n}\phi([(a_{i,n})_{i}])\leq \phi(\sup_{n}[(a_{i,n})_{i}])$, and thus $\phi(\sup_n [(a_{i,n})_i])=\sup_n\phi([(a_{i,n})_i]$, as required.
	
	By construction, $\phi$ is a left-inverse for $\varphi_0=\varphi|_{H}$. By definition of $\varphi$ and since $\phi$ preserves suprema of increasing sequences, it follows that $\phi$ is a left-inverse for~$\varphi$.
\end{proof}

\begin{parag}[Hilbert $\mathrm{C}^*$-modules]\label{para:HilbCMod}
	For a \Cs{} $A$, we consider the class $\preCH (A)$ of countably generated Hilbert $A$-modules, see for instance \cite{MT05} for definitions and background. Let $H_A$ be the Hilbert $A$-module consisting of sequences 
	$(a_n)$ of elements in $A$ such that $\sum_{n=1}^{\infty} a_n^*a_n$ is norm-converging in $A$. Note that $H_A$ is the 
	Hilbert $A$-module completion of the $A$-module $A^{(\N)}$.
	By Kasparov's Theorem (see e.g. \cite[Theorem 1.4.2]{MT05}) each countably generated Hilbert $A$-module is 
	isometrically isomorphic to a complemented $A$-submodule of $H_A$.  
	
	Denote by $\mathcal K (X)$ the \Cs{} of compact operators on a Hilbert $A$-module $X$.
	If $X\subseteq Y$ are 
	Hilbert $A$-modules, we say that $X$ is {\it compactly contained} in $Y$ if there exists a self-adjoint compact operator 
	$\theta \in \mathcal K (Y)$ such that $\theta|_X = \text{id}_X$. Given Hilbert $A$-modules 
	$X$ and $Y$, we say that $X$ is {\it Cuntz subequivalent}
	to $Y$, written $X\precsim Y$, if each Hilbert submodule $X_0$ of $X$ which is compactly contained in $X$ is isometrically isomorphic to a Hilbert module $Y_0$ which is compactly contained in $Y$. We say that $X$ and $Y$ are {\it Cuntz equivalent}, 
	written $X\sim Y$, if $X\precsim Y$ and $Y\precsim X$. The semigroup $\CH (A)$ is then the semigroup of Cuntz 
	equivalence classes of countably generated Hilbert $A$-modules, endowed with the operation induced by the direct 
	sum of Hilbert $A$-modules. 
	
	It was shown in \cite{Coward2008} that there is an isomorphism $\Cu (A)\cong \CH (A)$ in the category $\Cu$. This isomorphism 
	sends the class of a positive element $a$ in $A\otimes \mathcal K$ to $\ol{a(H_A)}$, where we use the isomorphism 
	$A\otimes \mathcal K \cong \mathcal K (H_A)$, see e.g. \cite[Proposition 3.15(iii)]{APT2011}.
	
	Let $A$ be a \Cs{}. Then we have, on the one hand, an isomorphism $\gamma_c\colon \Cu (A) \cong \CH (A)$,
	and on the other hand an isomorphism  
	$\gamma_a\colon \SR (A) \cong \CP (A)$ by \autoref{thm:nonunitalSRCP}, where $\CP(A)$ is built from the 
	category $\preCP (A)$ of countably generated projective unital right $A^+$-modules $P$ such that $P=PA$, 
	see \autoref{para:CPRNonUni}. 
	
	Hence there is a unique morphism $\tilde{\phi} \colon \CP (A)\to \CH (A)$ making commutative the following diagram:
	\[\xymatrix{
		\Cu (A) \ar@{->}[r]^{\gamma_c} & \CH (A)   \\
		\SR (A)  \ar@{->}[r]^{\gamma_a} \ar@{->}[u]^{\phi}  & \CP (A) \ar@{->}[u]^{\tilde{\phi}} } \]
	namely $\tilde{\phi} = \gamma_c\circ \phi \circ \gamma_a^{-1}$.
\end{parag}

\begin{proposition}
	\label{prop:closure-is-OK}
	Let $P$ be an object in $\CP (A)$, and let
	$x_n \in M_{\infty}(A)$ be a sequence such that $x_{n+1}x_n = x_n$ for each $n\ge 1$ and $P\cong Q:=\bigcup_{i=1}^{\infty} x_nA^{(\N)}$. We then have 
	$$\tilde{\phi} ([P]) = \tilde{\phi} ([Q])= [\ol{Q}],$$
	where $\ol{Q}$ is the Hilbert $A$-module obtained by taking the closure of $Q$ in $H_A$.  
\end{proposition}

\begin{proof}
	We first observe that a sequence $(x_n)$ as in the statement always exists 
	by \autoref{cor:special-rep}. Given such a sequence $(x_n)$, we have $x_n H_A \subseteq x_{n+1} H_A$ 
	and in particular $\ol{x_nH_A}\subseteq \ol{x_{n+1}H_A}$, so that by \cite[Proposition 4.12]{APT2011} we have
	$$\sup_n [\ol{x_nH_A}] = [\ol{\cup _{n=1}^{\infty} \ol{x_nH_A}}]$$
	in $\CH (A)$.
	On the other hand, 	by \cite[Lemma 4.10]{APT2011}, we have
	$$\ol{x_n^*x_nH_A} \cong \ol{x_nx_n^*H_A} = \ol{x_nH_A}$$ 
	for all $n\ge 1$. Therefore we get	$\gamma_c ([x_n^*x_n]) = [\ol{x_n H_A}]$. Using that $\gamma_c $ preserves suprema of increasing     
	sequences, we have
	\begin{align*}
		\gamma_c\circ \phi \circ \gamma_a^{-1} ([Q]) & = \gamma_c \circ \phi ([(x_n)]) 
		= \gamma_c (\sup_n [x_n^*x_n]) = \sup_n \gamma_c ([x_n^*x_n]) \\
		& = \sup_n [\ol{x_nH_A}] 
		=  [\ol{\cup _{n=1}^{\infty} \ol{x_nH_A}}] 
		= [\ol{Q}],
	\end{align*}	
	as desired.
\end{proof}

\section{Nearly simple domains}
\label{sec:nearly}

In this section we study nearly simple domains, a class of rings where one can explicitly compute the monoid $\W (R)$; see \autoref{parag:J1Simp}.

As we will prove in \autoref{prop}, the Jacobson radical $J$ of any nearly simple domain $R$ is always weakly $s$-unital, although it is not $s$-unital in general. The invariants $\RLambda (J)$ and $\SR (J)$ are computed in \autoref{thm:computeSJ} and \autoref{rmk:WJ} respectively.

\begin{parag}[Uniserial domains and nearly simple domains]
	\label{parag:J1Simp} Recall that a module over a ring $R$ is \emph{uniserial} if its submodules are totally ordered by inclusion, and that the ring $R$ is said to be \emph{right uniserial} if it is uniserial as a right module over itself.
	
	One defines left uniserial rings analoguously, and says that $R$ is \emph{uniserial} if it is both right and left uniserial. Uniserial rings will be assumed to be unital throughout the section.
	
	Note that any right uniserial domain $R$ is a local ring. That is, $R$ has a unique maximal left ideal. The reader is referred to \cite{Fac1998} for a thorough exposition.
	
	Let $R$ be a uniserial domain. We will say that $R$ is a \emph{nearly simple domain} if $R$ is not simple and the only two-sided ideals of $R$ are $\{0\}$, $J(R)$ and $R$. 
	
	Given elements $r,s$ in a unital, uniserial ring $R$, it is well-known that $RrR =RsR$ if and only if there exist units $u,v\in R$ such that $r=usv$; see \cite[Lemma~4.2]{Pun2001}.
	
	In particular, if $R$ is a nearly simple domain, this implies that $J:=J(R)$ is a \emph{$1$-simple ring}, that is, for each $r,s\in J\setminus \{0\}$ there exist $a,b\in J$ such that $r=asb$ (see \cite{Cohn05}, where the concept of an $n$-simple ring is introduced for unital rings, for every $n\ge 1$). Indeed, applying \cite[Lemma~4.2]{Pun2001} to $r,s^3$, we obtain units $u,v\in R$ such that 
	\[
	r=us^3 v=(us)s(sv).
	\]
	The elements $a:=us$ and $b:=vs$ are in $J$ and satisfy the desired equality.
\end{parag}


As shown in \cite[Theorem 2.10]{AGOR2000}, every regular square matrix over an exchange separative ring can be diagonalized by using row and column elementary transformations. Since local rings are separative exchange rings, this applies in particular to any uniserial ring. We will see in \autoref{lem:matricesovernearlysimple} below that {\it all}  square matrices over a uniserial ring are equivalent to diagonal matrices. This will allow us to compute $\W (R)$ for a nearly simple domain $R$ in \autoref{thm:computeSJ}. Note that there exist artinian local commutative rings $R$ such that some $2\times 2$ matrices over $R$ are not diagonalizable (see \cite[Remark 2.12]{AGOR2000}).  

Let us denote by $E_n (R)$ the set of $n\times n$ elementary matrices.

\begin{lemma}\label{lem:matricesovernearlysimple}
	Let $R$ be a uniserial ring. Then, for each square matrix $A\in M_n(R)$ there exist elementary matrices $U,V\in E_n(R)$ such that $UAV$ is diagonal.
\end{lemma}

\begin{proof}
	We proceed by induction on $n$ and note that the  case $n=1$ is trivial. 
	
	Thus, let $n> 1$ be fixed and assume that we have proven the result for every $k\times k$ matrix with $k\leq n-1$.
	
	Take $A\in M_n (R)$. If $A$ has an invertible entry, we can move such entry to the position $(1,1)$ by means of elementary transformations. Further, since this entry is now invertible, there exist elementary matrices $U,V$ such that the product $A'=UAV$ satisfies $A'(1,i)= A' (i,1)= 0$ for all $i>1$. The desired result now follows by induction.
	
	Thus, it remains to consider the case $A\in M_n (J)$, where we will show by induction on $k$ that there exist elementary matrices $U_k, V_k\in E_n (R)$ such that the product 
	\[
	B_k = U_k A V_k
	\]
	satisfies $B_k (i,j)=0$ for every pair $(i,j)$ such that $i\leq k$ and $i\neq j$. That is, $B_k$ is of the form 
	\[
	B_k=
	\left(\begin{array}{c}
		\begin{array}{ccc|ccc}
			B_k (1,1) & & 0 & & & \\
			& \ddots & & & 0 & \\
			0 & & B_k (k,k) & & &
		\end{array}
		\\ \hline
		\\
		C_k
		\\
	\end{array}
	\right )
	\]
	for some matrix $C_k$.
	
	If $k=1$, use that $R$ is uniserial to find $i\geq 1$ such that $A(1,j)R\subseteq A(1,i)R$ for every $j$. Using elementary transformations, we may assume that $i=1$. This shows that $A$ can be transformed into a matrix $B_1$ satisfying the required conditions. 
	
	Now fix $k< n$ and assume that we have proven the result for every $k'\leq k$. In particular, we can find $U_k, V_k\in E_k (R)$ such that $U_k A V_k = B_k$.
	
	Then, for every $i\leq k$, we either have that $RB_k (k+1, i)\subseteq RB_k (i,i)$ or $RB_k (i,i)\subseteq RB_k (k+1, i)$. Performing elementary row operations, we may assume that $B_k (k+1, i)=0$ whenever $RB_k (k+1, i)\subseteq RB_k (i,i)$.
	
	Let $k'$ be such that $B_k (k+1, i)R\subseteq B_k (k+1,k')R$ for every $i$. We may assume that $B_k (k+1, k') \neq 0$, since we are done otherwise.
	
	If $k'\geq k+1$, we can perform elementary column operations in order to get $k'=k+1$. Using once again column operations, we obtain a matrix of the form 
	\[
	\left(\begin{array}{c}
		\begin{array}{cccc|ccc}
			B_k (1,1) & &  & 0 & & & \\
			& \ddots & & & & 0 & \\
			& & B_k (k,k) & & & &\\
			0 & & & B_k (k+1,k+1) & & &
		\end{array}
		\\ \hline
		\\
		C_{k+1}
		\\
	\end{array}
	\right )
	\]
	for some $C_{k+1}$, as desired.
	
	Finally, assume that $k'\leq k$. Since $B_k (k+1, k')\neq 0$, we have
	\[
	RB_k (k',k')\subseteq RB_k (k+1, k').
	\]
	
	Let $B'$ be the matrix resulting from adding a multiple of the $(k+1)$-th row to the $k'$-th row in such a way that $B' (k',k')=0$.
	
	Performing elementary column operations, we obtain yet another matrix $B''$ with $B'' (k+1, i)=0$ for every $i\neq k'$.
	
	Swapping the $k'$-th row with the $(k+1)$-th row, we find a matrix $B_k'$ with $B_k' (i,j)=0$ for every $(i,j)$ such that $i\leq k$ and $i\neq j$. Further, $B_k'$ has at least one more zero than $B_k$ in the $(k+1)$-th row. Proceeding by induction, we get matrices $U_{k+1}$, $V_{k+1}$ and $B_{k+1}$ with the desired properties. This finishes the inductive argument.
	
	Since $B_n$ is a diagonal matrix, the matrices $U:=U_n$ and $V:= V_n$ satisfy the required conditions. This finishes the proof.
\end{proof}

Let $(M,\le)$ be a partially ordered monoid, and let $I$ be a submonoid of $M$. Recall that $I$ is said to be an \emph{o-ideal} of $M$ if $I$ is hereditary for $\le$, that is, if whenever $x\le y$ with $y\in I$ we have $x\in I$.

\begin{proposition}\label{prop} 
	Let $R$ be a nearly simple domain. Then $J(R)$ is a weakly $s$-unital ring,
	and there is an order-embedding of $\W (J(R))$ into an o-ideal of $\W (R)$. 
\end{proposition}
\begin{proof}
	Set $J:=J(R)$ and take $A\in M_n(J)$ for some $n\ge 1$. We have to show that  there exist matrices $X,Y\in M_n(J)$ such that $A=XAY$. 
	By \autoref{lem:matricesovernearlysimple} there exist elementary matrices $U,V\in E_n(R)$ such that $UAV=D$, where $D$ is diagonal matrix in $M_n(R)$. Note that, since all the entries of $A$ belong to $J$, we have $D\in M_n(J)$.
	
	Further, we know from \autoref{parag:J1Simp} that $J$ is a $1$-simple ring. Thus, we can find diagonal matrices $Z,T\in M_n(J)$ satisfying $D=ZDT$. This implies 
	\[
	A= U^{-1} D V^{-1} = U^{-1} Z D TV^{-1} = (U^{-1}ZU)A(VTV^{-1})
	\]
	and, consequently, the matrices $X:=U^{-1}ZU$ and $Y:=VTV^{-1}$ are in $M_n (J)$ and satisfy $A=XAY$, as desired.
	
	It follows from \autoref{par:WR} that we can form the semigroup $\W (J)$, which is a positively ordered monoid.
	
	The inclusion map $J\to R$ induces a positively ordered monoid-morphism $\W (J)\to \W (R)$. To see that it is an order-embedding, let $A,B\in M_n(J)$ and assume that $A\precsim_1 B$ in $M_n (R)$. Let $P,Q\in M_n (R)$ be such that $A=PBQ$. Using the first part of the proposition, we obtain elements $X,Y\in M_n(J)$ such that $A=XAY$, and hence $A= (XP)B(QY)$.
	
	This shows that $A\precsim_1 B$ in $M_n(J)$ and, therefore, that the map $\W (J)\to \W (R)$ is an order-embedding.
	
	Identifying $\W (J)$ with its image, it is readily checked that $\W (J)$ is an o-ideal of $\W (R)$.
\end{proof}


\begin{theorem}\label{thm:computeSJ}
	Let $R$ be a nearly simple domain, and let $J$ be its  Jacobson radical. Then,
	\begin{enumerate}[{\rm(i)}]
		\item $\W (J)\cong \N$, with its usual order.
		\item $\W (R) \cong \N\times \N$, with the order 
		\[
		(r',s')\le (r,s) :\iff r'+s' \le r+s \quad \mathrm{ and }\quad  r'\le r.
		\]
	\end{enumerate}
\end{theorem}
\begin{proof}
	Take $A\in M_n(R)$. Using \autoref{lem:matricesovernearlysimple}, we find invertible matrices $U,V$ and a diagonal matrix $D\in M_n(R)$ such that $UAV=D$. We may assume that $D$ is of the form
	\[
	D=\text{diag} (d_1,\dots , d_r, d_{r+1}, \dots , d_{r+s}, 0,\dots , 0)
	\]
	for some $d_1,\dots, d_r\in R\setminus J$ and $d_{r+1},\dots , d_{r+s}\in J\setminus \{ 0 \}$. Let
	\[
	\psi\colon M_\infty (R)\to \N\times \N
	\]
	be the map defined by $\psi (A) := (r,s)$.
	
	To see that $\psi (A)$ does not depend on the choice of $U$ and $V$, set 
	\[
	C : =R/d_{r+1}R\oplus \ldots \oplus R/d_{r+s}R\oplus R^{n-r-s}
	\]
	and consider the commutative diagram
	\[
	\xymatrix{
		R^n  \ar@{->}[r]^{A}\ar@{<-}[d]_{V}^{\cong}  & R^n \ar@{->}[r] \ar@{->}[d]_{U}^{\cong}  &  R^n/AR^n \ar@{->}[d]^{\cong} \ar@{->}[r] & 0 \\
		R^n \ar@{->}[r]^{D} & R^n \ar@{->}[r] & C \ar@{->}[r] & 0  } 
	\]
	
	Choosing  $a\in J\setminus \{0\}$, we have $C\cong (R/aR)^s\oplus R^{n-r-s}$ (by \cite[Lemma~4.2]{Pun2001}). Thus, for any choice of invertible matrices $U',V'$ and diagonal matrix $D'$ such that $D'=U'AV'$ with ranks $(r',s')$, one gets 
	\[
	(R/aR)^s\oplus R^{n-r-s}\cong (R/aR)^{s'}\oplus R^{n-r'-s'}.
	\]
	At this point we can use Puninski's Theorem \cite[Theorem 9.19]{Fac1998} asserting that every finitely presented right module $M$ over a uniserial ring is the direct sum of cyclic uniserial modules, and any two decompositions of $M$ as direct sums of cyclic modules are isomorphic. Using this result we immediately deduce that $s=s'$ and $n-r-s= n-r'-s'$, and thus $r= r'$.   
		
	Next we show that $A\precsim_1 B$ implies $\psi (A) \le \psi(B)$, where recall that
	\[
	(r',s')\le (r,s) :\iff r'+s' \le r+s \quad \mathrm{ and }\quad  r'\le r.
	\]
	
	Thus, let $A,B$ be such that $A\precsim_1 B$, and write $\psi (A)= (r_A, s_A)$ and $\psi (B)= (r_B , s_B)$. We may assume that $A,B\in M_n(R)$, and that there are matrices $X,Y\in M_n(R)$ such that $A=XBY$. 
	
	Let $\pi$ be the quotient map $R\to R/J$. We have 
	\[
	\pi (A)=\pi (X)\pi(B)\pi (Y)\text{ in }M_n(R/J)
	\]
	and, therefore, 
	\[
	r_A=\text{rank}_{R/J} (\pi (A))\leq \text{rank}_{R/J} (\pi (B))=r_B.
	\]
	
	Let us now prove that $\psi (BY)\le \psi (B)$ for all $B,Y\in M_n(R)$. Following the notation above, we write 
	$\psi (B)= (r_B,s_B)$ and $\psi (BY)= (r_{BY},s_{BY})$.
	
	The previous argument shows that $r_{BY}\le r_B$, so it remains to check that 
	\[
	r_{BY}+s_{BY}\le r_B+s_B.
	\]
	
	Since $R^n/BR^n$ is a quotient of $R^n/(BY)R^n$, we obtain a surjective module homomorphism
	\[
	(R/aR)^{s_{BY}}\oplus R^{n-r_{BY}-s_{BY}} \longrightarrow R^{n-r_B-s_B},
	\]
	where $a\in J \setminus \{0\}$. Using that $R^{n-r_B-s_B}$ is free, we find a right $R$-module $M$ such that
	\[
	(R/aR)^{s_{BY}}\oplus R^{n-r_{BY}-s_{BY}} \cong R^{n-r_B-s_B}\oplus M.
	\]
	By Puninski's Theorem \cite[Theorem 9.19]{Fac1998} we have that $n-r_{BY}-s_{BY}\ge n-r_B-s_B$. Thus, we get $r_{BY}+s_{BY}\le r_B+s_B$, as desired.
	
	Note that, by symmetry, we have $\psi (XB)\le \psi (B)$ for all $X,B\in M_n(R)$. Consequently, one gets $\psi (A)\le \psi (B)$ whenever $A\precsim_1 B$.
	
	Conversely, it is also easy to see (by looking at their associated diagonal matrices) that $A\precsim_1 B$ whenever $\psi (A)\le \psi (B)$. 
	
	Thus, $\psi$ induces an order-isomorphism from $\W (R)$ to $\N\times\N$ with the stated order. This shows (ii).
	
	To see (i), note that the image of $\W (J)$ through this order-isomorphism corresponds to $0\times\N\cong\N$. The induced order in this submonoid corresponds to the usual order.
\end{proof}

\begin{remark}\label{rmk:ComplNonComp}
	Let $R$ be a nearly simple domain, and let $J$ be its Jacobson radical. Then, $V(J)=0$ and $\W (J)=\N$ by \autoref{thm:computeSJ}.
	
	Thus, any element $x\in \W (J)$ satisfies that, whenever $x\leq y$, there exists $c$ with $x+c=y$. However, there are no nonzero elements in $V(J)$.
	
	In connection with \autoref{lem:alg-order}, the above shows that elements in $W(J)$ can be complented; though, they are not coming from $V(J)$.
\end{remark}

\begin{remark}\label{rmk:WJ}
	It follows from \autoref{thm:computeSJ} above that, if $R$ is a nearly simple domain, the monoid $\WLambda (J)$ associated to its Jacobson radical $J=J(R)$ is indistinguishable from $\WLambda (D)$ with $D$ a division ring; see \autoref{dfn:IntR}.
	
	However, note that every sequence $(x_n)$ defining an element in $\SR (J)$ induces a countably generated projective module $P$ over $R$ such that $P=PJ(R)$. Thus, we have $P=0$. This shows that $\SR (J)\cong 0$.
	
	On the other hand, $\SR (D)\cong \ol{\N}$ for any division ring. Consequently, $\SR (R) $ distinguishes these two families of rings.
\end{remark}



\begin{thebibliography}{10}
	
	\bibitem{AntAraBosPerVil23arX:ContandIde}
	R.~Antoine, P.~Ara, J.~Bosa, F.~Perera, and E.~Vilalta.
	\newblock Continuity and the ideal lattice of the {C}untz semigroup of a ring.
	\newblock In preparation, 2023.
	
	\bibitem{AntBosPer11}
	R.~Antoine, J.~Bosa, and F.~Perera.
	\newblock Completions of monoids with applications to the {C}untz semigroup.
	\newblock {\em Internat. J. Math.}, 22(6):837--861, 2011.
	
	\bibitem{APT-Memoirs2018}
	R.~Antoine, F.~Perera, and H.~Thiel.
	\newblock Tensor products and regularity properties of {C}untz semigroups.
	\newblock {\em Mem. Amer. Math. Soc.}, 251(1199):viii+191, 2018.
	
	\bibitem{AAlgCol2004}
	P.~Ara.
	\newblock Rings without identity which are {M}orita equivalent to regular
	rings.
	\newblock {\em Algebra Colloq.}, 11(4):533--540, 2004.
	
	\bibitem{AGOR2000}
	P.~Ara, K.~R. Goodearl, K.~C. O'Meara, and R.~Raphael.
	\newblock {$K_1$} of separative exchange rings and {$C^\ast$}-algebras with
	real rank zero.
	\newblock {\em Pacific J. Math.}, 195(2):261--275, 2000.
	
	\bibitem{AraGooPar02}
	P.~Ara, K.~R. Goodearl, and E.~Pardo.
	\newblock {$K_0$} of purely infinite simple regular rings.
	\newblock {\em $K$-Theory}, 26(1):69--100, 2002.
	
	\bibitem{APP2000}
	P.~Ara, E.~Pardo, and F.~Perera.
	\newblock The structure of countably generated projective modules over regular
	rings.
	\newblock {\em J. Algebra}, 226(1):161--190, 2000.
	
	\bibitem{APT2011}
	P.~Ara, F.~Perera, and A.~Toms.
	\newblock {$K$}-theory for operator algebras. {C}lassification of
	{$C^*$}-algebras.
	\newblock In {\em Aspects of operator algebras and applications}, volume 534 of
	{\em Contemp. Math.}, pages 1--71. Amer. Math. Soc., Providence, RI, 2011.
	
	\bibitem{AGPS2010}
	G.~Aranda~Pino, K.~R. Goodearl, F.~Perera, and M.~Siles~Molina.
	\newblock Non-simple purely infinite rings.
	\newblock {\em Amer. J. Math.}, 132(3):563--610, 2010.
	
	\bibitem{Bergman-Diamond}
	G.~M. Bergman.
	\newblock The diamond lemma for ring theory.
	\newblock {\em Adv. in Math.}, 29(2):178--218, 1978.
	
	\bibitem{BroLin23arX}
	L.~G. Brown and H.~Lin.
	\newblock Projective {H}ilbert modules and sequential approximation.
	\newblock Preprint (arXiv:2301.04247v1 [math.OA]), 2023.
	
	\bibitem{BroCiu09}
	N.~P. Brown and A.~Ciuperca.
	\newblock Isomorphism of {H}ilbert modules over stably finite {$C^*$}-algebras.
	\newblock {\em J. Funct. Anal.}, 257(1):332--339, 2009.
	
	\bibitem{Cohn05}
	P.~M. Cohn.
	\newblock On {$n$}-simple rings.
	\newblock {\em Algebra Universalis}, 53(2-3):301--305, 2005.
	
	\bibitem{Coward2008}
	K.~Coward, G.~Elliott, and C.~Ivanescu.
	\newblock The {C}untz semigroup as an invariant for {$C^*$}-algebras.
	\newblock {\em J. Reine Angew. Math.}, 623:161--193, 2008.
	
	\bibitem{Cun78DimFct}
	J.~Cuntz.
	\newblock Dimension functions on simple {C}*-algebras.
	\newblock {\em Math. Ann.}, 233(2):145--153, 1978.
	
	\bibitem{Fac1998}
	A.~Facchini.
	\newblock {\em Module Theory: Endomorphism Rings and Direct Sum Decompositions
		in Some Classes of Modules}, volume 167 of {\em Progr. Math.}
	\newblock Birkh\"{a}user, 1998.
	
	\bibitem{FH2000}
	A.~Facchini and D.~Herbera.
	\newblock ${K}_0$ of a semilocal ring.
	\newblock {\em J. of Algebra}, 255:47--69, 2000.
	
	\bibitem{GarPer2022}
	E.~Gardella and F.~Perera.
	\newblock The modern theory of {C}untz semigroups of {C}*-algebras.
	\newblock Preprint (arXiv:2212.02290v2 [math.OA]), 2022.
	
	\bibitem{GieHof+03Domains}
	G.~Gierz, K.~H. Hofmann, K.~Keimel, J.~D. Lawson, M.~Mislove, and D.~S. Scott.
	\newblock {\em Continuous lattices and domains}, volume~93 of {\em Encyclopedia
		of Mathematics and its Applications}.
	\newblock Cambridge University Press, Cambridge, 2003.
	
	\bibitem{vnrr}
	K.~R. Goodearl.
	\newblock {\em von {N}eumann regular rings}.
	\newblock Robert E. Krieger Publishing Co., Inc., Malabar, FL, second edition,
	1991.
	
	\bibitem{Goodirectlim}
	K.~R. Goodearl.
	\newblock Leavitt path algebras and direct limits.
	\newblock In {\em Rings, modules and representations}, volume 480 of {\em
		Contemp. Math.}, pages 165--187. Amer. Math. Soc., Providence, RI, 2009.
	
	\bibitem{HPCrelle}
	D.~Herbera and P.~P\v{r}\'{\i}hoda.
	\newblock Big projective modules over noetherian semilocal rings.
	\newblock {\em J. Reine Angew. Math.}, 648:111--148, 2010.
	
	\bibitem{HPTrans}
	D.~Herbera and P.~P\v{r}\'{\i}hoda.
	\newblock Infinitely generated projective modules over pullbacks of rings.
	\newblock {\em Trans. Amer. Math. Soc.}, 366(3):1433--1454, 2014.
	
	\bibitem{Herbera2014}
	D.~Herbera and P.~P\v{r}\'{i}hoda.
	\newblock Reconstructing projective modules from its trace ideal.
	\newblock {\em J. Algebra}, 416:25--57, 2014.
	
	\bibitem{HunLi21}
	T.~F. Hung and H.~Li.
	\newblock Malcolmson semigroups.
	\newblock {\em J. Algebra}, 623:193--233, 2023.
	
	\bibitem{Kasparov80}
	G.~G. Kasparov.
	\newblock Hilbert {C}*-modules: Theorems of {S}tinespring and {V}oiculescu.
	\newblock {\em J. Operator Theory}, 4(1):133--150, 1980.
	
	\bibitem{Lam1999}
	T.~Y. Lam.
	\newblock {\em Lectures on {M}odules and {R}ings}, volume 199 of {\em Graduate
		Texts in Mathematics}.
	\newblock Springer, 1999.
	
	\bibitem{MT05}
	V.~M. Manuilov and E.~V. Troitsky.
	\newblock {\em Hilbert {$C^*$}-modules}, volume 226 of {\em Translations of
		Mathematical Monographs}.
	\newblock American Mathematical Society, Providence, RI, 2005.
	\newblock Translated from the 2001 Russian original by the authors.
	
	\bibitem{Puninski2007}
	W.~McGovern, G.~Puninski, and P.~Rothmaler.
	\newblock When every projective module is a direct sum of finitely generated
	modules.
	\newblock {\em J. Algebra}, 315:454--481, 2007.
	
	\bibitem{Per97}
	F.~Perera.
	\newblock The structure of positive elements for {$C^*$}-algebras with real
	rank zero.
	\newblock {\em Internat. J. Math.}, 8(3):383--405, 1997.
	
	\bibitem{Prihoda2007}
	P.~P{\v r}{\' i}hoda.
	\newblock Projective modules are determined by their radical factors.
	\newblock {\em J. Pure Appl. Algebra}, 210:827--835, 2007.
	
	\bibitem{Pun2001}
	G.~Puninski.
	\newblock Some model theory over a nearly simple uniserial domain and
	decompositions of serial modules.
	\newblock {\em J. Pure Appl. Algebra}, 163(3):319--337, 2001.
	
	\bibitem{ThiVil22DimCu}
	H.~Thiel and E.~Vilalta.
	\newblock Covering dimension of {C}untz semigroups.
	\newblock {\em Adv. Math.}, 394:44~p., Article No.~108016, 2022.
	
	\bibitem{Tom08ClassificationNuclear}
	A.~S. Toms.
	\newblock On the classification problem for nuclear {$C^*$}-algebras.
	\newblock {\em Ann. of Math. (2)}, 167(3):1029--1044, 2008.
	
	\bibitem{Weh96}
	F.~Wehrung.
	\newblock Monoids of intervals of ordered abelian groups.
	\newblock {\em J. Algebra}, 182(1):287--328, 1996.
	
\end{thebibliography}
\end{document}